\documentclass[a4paper,10pt]{article}

\usepackage{hyperref}
\usepackage{color}
\usepackage[english]{babel}

\usepackage[utf8]{inputenc}

\usepackage[T1]{fontenc}
\usepackage{setspace}

\usepackage{enumerate}
\usepackage{microtype}
\usepackage{cite}
\usepackage[a4paper,top=2cm,bottom=2cm,left=2.5cm,right=2.5cm,bindingoffset=5mm]{geometry}

\usepackage{tikz}
\usetikzlibrary{decorations.pathmorphing}

\usepackage{amsmath}
\usepackage{amsfonts}
\usepackage{amssymb}
\usepackage{amsthm}
\usepackage{braket}
\usepackage{marvosym}
\usepackage{eufrak}
\usepackage{esint}

\theoremstyle{plain}
\newtheorem{teo}{Theorem}[section]
\newtheorem{lem}[teo]{Lemma}

\newtheorem{prop}[teo]{Proposition}
\newtheorem*{teo*}{Theorem}

\theoremstyle{remark}
\newtheorem{oss}[teo]{Remark}

\theoremstyle{definition}

\numberwithin{equation}{section}

\newcommand{\N}{\mathbb{N}}
\newcommand{\R}{\mathbb{R}}

\newcommand{\tr}{\text{Tr}}
\renewcommand{\a}{\alpha}

\renewcommand{\l}{\lambda}

\renewcommand{\d}{\delta}
\renewcommand{\c}{\gamma}

\newcommand{\diw}{\mathrm{div}}

\let\temp\phi
\let\phi\varphi
\let\varphi\temp
\newcommand{\uu}{u}
\newcommand{\EE}{\mathcal{E}}
\newcommand{\wx}{\widehat{x}}
\newcommand{\GG}{\mathcal{G}}
\newcommand{\RR}{R_a}

\title{An epiperimetric inequality\\ for the lower dimensional obstacle problem}
\author{Francesco Geraci}
\date{}
\begin{document}
\maketitle

\begin{abstract}
In this paper we give a proof of an \emph{epiperimetric inequality} in the setting of the lower dimensional obstacle problem.  
The inequality was introduced by Weiss (Invent. Math., 138 (1999), no. 1, 23–50) for the classical obstacle problem and
 has striking consequences concerning the regularity of the free-boundary.
Our proof follows the approach of Focardi and Spadaro 
(Adv. Differential Equations 21 (2015), no 1-2, 153-200.) 
which uses an homogeneity approach and a $\Gamma$-convergence analysis.
\end{abstract}

\section*{Introduction}\addcontentsline{toc}{section}{Introduction}
The obstacle problem consists in finding the minimizer of a suitable energy among all
functions, with fixed boundary data, constrained to lie above a given obstacle.
The obstacle can live in the whole domain or on
a surface of codimension one, these cases are denoted by the classical obstacle and the lower dimensional obstacle 
(or the thin obstacle) respectively.
In this paper we analyse a particular case of a lower dimensional obstacle
where the obstacle is laid in a hyperplane of
the domain and the energies are the weighted versions of Dirichlet energy. 
The motivation for studying lower dimensional obstacle problems has roots in many applications. 
There are examples in physics, mechanics, biology and financial mathematics and many prime examples can be found in 
\cite{BG90, CS05, CV10, CT04, DL76, Fried, KN77, KS, PSU, Rod}.\\
In this paper we consider the energy
\begin{equation}\label{i:e:EE}
 \EE(v):=\int_{B_1^+} |\nabla v|^2\,x_n^a\,dx,
\end{equation}
and minimize $\EE$ among all functions in the class of admissible functions
\begin{equation}
 \mathfrak{A}_g:=\{v\in H^1(B_1^+, \mu_a)\,:\, v\geq 0\,\it{on}\, B'_1, v=g\,\it{on}\, (\partial{B_1})^+ \},
\end{equation}
where $H^1(A, \mu_a)$ is the weighted Sobolev Space and $\mu_a$ is the measure $\mu_a:=|x_n|^a\,\mathcal{L}^{n}\llcorner B_1$
with $a\in(-1,1)$. 
In what follows we will extend automatically every functions in $\mathfrak{A}_g$ by even symmetry with respect to $\{x_n=0\}$
and for convenience we will indicate any points $x\in \R^n$ as $x=(\wx,x_n)\in \R^{n-1}\times \R$.

By the direct method of calculus of variations it easy to prove the existence and uniqueness of the minimum of \eqref{i:e:EE} on $\mathfrak{A}_g$.
Let $\uu:=\mathrm{min}_{\mathfrak{A}_g}\,\EE$, we note that $\uu$ satisfies the following Euler–Lagrange equations:  
\begin{align}\label{POF'}
		 \left\{  \begin{array}{ll}
			  \uu(\wx,0)\geq 0 & 						\quad \wx\in B_1\\ 
			  \uu(\wx,x_n)=\uu(\wx,-x_n) &	\\ 
			  \diw(|x_n|^a\,\nabla \uu(\wx,x_n))=0 &			\quad x\in B_1\setminus\{(\wx,0)\,:\,\uu(\wx,0)=0\}\\ 
			  \diw(|x_n|^a\,\nabla \uu(\wx,x_n))\leq 0 &    		\quad x\in B_1\,\,\textit{in distributional sense.}
	  		\end{array} \right.
\end{align}
We denote by $\Gamma(\uu):=\partial \{(\wx,0)\in B'_1\,:\, \uu(\wx,0)=0\}\cap B'_1$ the free-boundary of $\uu$.

In order to establish the regularity of the solution $\uu$ and its free-boundary a fundamental tool is
the Almgren \emph{frequency} type function (see \cite{ACS} for $a=0$). For all points $x_0\in \Gamma(u)$:
\begin{equation}\label{i:e:Na}
  N_a^{x_0}(r,\uu):=\frac{r\int_{B_r(x_0)}|\nabla \uu|^2\,d\mu_a}{\int_{\partial B_r(x_0)} \uu^2\,|x_n|^a\,d\mathcal{H}^{n-1}}.
\end{equation}
Caffarelli and Silvestre \cite{CS} proved the monotonicity of function $r\mapsto N_a^{x_0}(r,\uu)$ and some of its properties 
such as,  
the property of being constant over all homogeneous functions; the two authors and Salsa \cite{CSS} established 
the property of the frequency function of being bigger than $1+s$,
where $s:=\frac{1-a}{2}$ is the exponent of the fractional Laplacian
of the trace, on $\R^{n-1}\times \{0\}$, of a global solution of \eqref{POF'}  (see Section~\ref{s:Preliminary}). 

Caffarelli, Salsa and Silvestre in \cite{CSS} proved the optimal regularity of the solution:
$u\in Lip(B_1)$, $\nabla_{\wx} u$ is one-sided $C^s$, or rather $\nabla_{\wx} u \in C^s(B^{\pm}_1\cup B'_1)$, 
and the weighted normal derivative $|x_n|^a \partial_n u$ is  $C^a(B^{\pm}_1\cup B'_1)$ for all $0<a<1-s$.  
In the particular case $s=1/2$, 
Athanasopoulos and Caffarelli \cite{AC} had already proven that $u\in C^{1,\frac{1}{2}}(B^{\pm}_1\cup B'_1)$. 

Thanks to the monotonicity of the functional \eqref{i:e:Na} it is possible to define the frequency of $u$ in $x_0$ as $N_a^{x_0}(0^+,u):=\lim_{r\to 0^+}N_a^{x_0}(r,u)$
and to distinguish the points in $\Gamma(u)$ respect to their frequencies.
The free-boundary $\Gamma(u)$ can be split as:
$$\Gamma(u)=Reg(u)\cup Sing(u) \cup Other(u).$$ 

The points of the subset $Reg(u)$ are called \emph{regular points}, they are the points of the free-boundary with 
least frequency i.e. $1+s$; we will denote $Reg(u)$ as $\Gamma_{1+s}(u)$. Caffarelli, Salsa and Silvestre in \cite{CSS} 
proved that $\Gamma_{1+s}(u)$ is locally a $C^{1,\a}$ $(n-1)$-submanifold. 
In the case $s=1/2$ the regularity of $\Gamma_{1/2}$ was already proved by Athanasopoulos, Caffarelli and Salsa in \cite{ACS}, 
while Focardi and Spadaro \cite{FS} and Garofalo, Petrosyan and Smit Vega Garcia in \cite{GPS16} gave alternative proofs of regularity
using an epiperimetric inequality (see Theorem~\ref{i:t:epiperimetric inequality}). 

The points of the subset $Sing(u)$ are called \emph{singular points} and are the points of the free-boundary with frequency $2m$ with $m\in \N$,
equivalently their contact sets have density zero with respect to $\mathcal{H}^n$. 
In the case $s=1/2$ Garofalo and Petrosyan \cite{GP09} prove that $Sing(u)$ is contained in a countable union of $C^1$ submanifold. 
Very recently Garofalo and Ros-Oton \cite{GR17} extended the result in \cite{GP09} for $s\in (0,1)$. 

The subset $Other(u)$ is the complement of $Reg(u)\cup Sing(u)$ in $\Gamma(u)$. 
Recently Focardi and Spadaro \cite{FS16} gave a complete description of the subset $Sing(u)\cup Other(u)$ 
up to a set of $\mathcal{H}^{n-2}$-measure zero. This result is new also in the framework of the Signorini problem, 
i.e. in the case $s=1/2$, and it is obtained by a combination of analytical and geometric measure theory arguments.\\

The goal of this paper is give an alternative proof of the regularity of $\Gamma_{1+s}(u)$ given by 
Caffarelli, Salsa and Silvestre in \cite{CSS}. 
Our proof use an epiperimetric inequality and its consequences. We extend the result proved by Focardi and Spadaro in \cite{FS} 
in the case $s\in (0,1)$.
The two authors outline the presence in their proof of two competing variational principles that contribute to the 
achievement of proof.\\

In order to enunciate the epiperimetric inequality 
we introduce a sequence of rescaled functions $u_{x_0,r}=\frac{u(x_0+rx)}{r^{1+s}}$ and an auxiliary energy ``à la Weiss''
\begin{equation}\label{e:e-weiss 1+s}
 W_{1+s}^{x_0}(r,u):=\frac{1}{r^{n+1}}\int_{B_r(x_0)}|\nabla u_r|^2\,d\mu_a -\frac{1+s}{r^{n+1}}\int_{\partial B_r(x_0)}|u_r|^2\,|x_n|^a\,d\mathcal{H}^{n-1},
\end{equation}
which is the sum of a volume energy and a boundary energy. 
We note that $1+s$, the frequency of points of the free-boundary examined, is the exponent of the 
scaling factor of sequence $u_{x_0,r}$ (see equation \eqref{d:ur 1+s}) and the coefficient of boundary energy. 
The existence of blow-ups is a consequence of a gradient estimate of rescaled function in $L^2(B_1,\mu_a)$; 
reasoning by contradiction, thanks to properties of the frequency and the optimal regularity of the solution we prove the 
$(1+s)$-homogeneity of blow-ups.
So, according to a result of classification by Caffarelli, Salsa and Silvestre \cite{CSS} we state the result of the classification
of $(1+s)$-homogeneous global solutions of the fractional obstacle, which constitute the following closed cone 
\begin{equation*}
 \mathfrak{H}_{1+s}:=\{\l\,h_e\,\,:\,\, e\in \mathbb{S}^{n-2}, \l\in [0,+\infty)\}\subset H^1_{loc}(\R^n,\mu_a),
\end{equation*}
with 
\begin{equation*}
 h_e(x):=\left(s^{-1}\wx\cdot e -\sqrt{(\wx\cdot e)^2+x_n^2}\right)\left(\sqrt{(\wx\cdot e)^2+x_n^2}+\wx\cdot e\right)^s.
\end{equation*}

The key result presented in this paper is an alternative proof 
(a first proof, with an extra hypothesis, was given by Garofalo, Petrosyan, Smit and Vega Garcia in \cite{GPPS}) 
of a Weiss' epiperimetric inequality for the fractional obstacle problem (cf. \cite[Theorem 1]{Weiss}). 
\begin{teo}[Epiperimetric inequality]\label{i:t:epiperimetric inequality}
 Let $\underline{0}\in \Gamma_{1+s}(u)$. There exists a dimensional constant $\kappa\in(0,1)$ such that if $c\in H^1(B_1,\mu_a)$ is a function $(1+s)$-homogeneous 
 for which $c\geq 0$ on $B'_1$ then
\begin{equation*}
 \inf_{v\in \mathfrak{A}_c} W_{1+s}^{\underline{0}}(1,v)\leq (1-\kappa)W_{1+s}^{\underline{0}}(1,c).
\end{equation*}
\end{teo}
Taking the epiperimetric inequality into account, Weiss proved this result in \cite{Weiss} in the classical obstacle case. 
Recently Garofalo, Petrosyan, Pop and Smit Vega Garcia \cite{GPPS} proved a similar epiperimetric inequality, with an extra hypothesis, 
for the fractional obstacle problem with drift in the case of $s\in (1/2,1)$.

In the case of obstacle $0$ and without drift our inequality is stronger. 
Indeed Garofalo \emph{et. al.} in \cite{GPPS} require an extra hypothesis of closeness 
between the function $c$ and a fixed blow-up limit.
We do not need such an assumption. 
On the other hand, due to homogeneity we can reduce to functions $c$ close to cone of global solutions $\mathfrak{H}_{1+s}$.

By contradicting the closeness assumption we obtain a quasi-minimality condition for a sequence of auxiliary functionals.
Using a $\Gamma$-convergence argument we inspect the $\Gamma$-limits of the sequence of auxiliary energies and analyse
their minimizer that represents the directions along which the epiperimetric inequality may fail.
Using a variational method we obtain that such minimizers show in the same time contradictory relationship with the cone 
$\mathfrak{H}_{1+s}$.

The epiperimetric inequality is a key ingredient to deduce the following estimate of the decay of energy:
\begin{equation}\label{e:decay}
 W_{1+s}^{x_0}(r,u)\leq C\,r^\gamma,
\end{equation}
where $C$ and $\gamma$ are positive constants.
Thanks to the decay estimate \eqref{e:decay} we prove a property of nondegeneration of solutions, 
from which we deduce that the blow-ups are nonzero. 
Proceeding as in \cite{FS}
we can prove the uniqueness of blow-ups and the regularity of $\Gamma_{1+s}(\uu)$; we state this results in
Proposition~\ref{p:uniqueness blowup 1+s} and Theorem~\ref{t:reg 1+s} respectively and do not prove them because 
they follow by the epiperimetric inequality and its consequences as in \cite[Proposition 4.8 and Proposition 4.10]{FS}\\

What follows is a summary of the structure of this paper: in section~\ref{s:3frequency} we introduce the frequency 
and its properties and define $\Gamma_{1+s}(u)$ the subset of free-boundary with low frequency.
In section~\ref{s:3existence} we prove the existence and $(1+s)$-homogeneity of blow-ups in the points in $\Gamma_{1+s}(u)$
and in section~\ref{s:3classification}, thanks to a result by \cite{CSS}, we characterize the $(1+s)$-homogeneous global 
solution of the fractional obstacle problem. 
Section~\ref{s:epiperimetric and consequences} is devoted to establish the epiperimetric inequality and its consequences 
in the framework of
the regularity of the free-boundary, a decay estimate of an auxiliary energy, the nondegeneracy of the solution 
and the uniqueness of the blow-ups.
In section~\ref{s:3reg}
we state the regularity of $\Gamma_{1+s}(u)$.

\section{Preliminary results}\label{s:Preliminary}
Let $\uu\in \mathrm{min}_{\mathfrak{A}_g}\,\EE$; we denote by $\Lambda(\uu)$ its coincidence set, 
$\Lambda(\uu):=\{\wx\in B'_1\,:\, \uu(\wx,0)=0\}$,
and by $\Gamma(\uu)$ its free-boundary $\Gamma(\uu):=\partial \Lambda(\uu)$ in $B'_1$ topology.

Caffarelli and Silvestre in \cite{CS} showed that the Euler-Lagrange equations of $\uu$ \eqref{POF'} are equivalent 
to the following equations:
\begin{align}\label{POF}
		 \left\{  \begin{array}{ll}
			  \uu(\wx,0)\geq 0 & 						\quad \wx\in(B'_1)^+\\ 
			  \diw(x_n^a\,\nabla \uu(\wx,x_n))=0 &				\quad x_n>0\\ 
			  \lim_{x_n\to0^+} x_n^a\partial_n\uu(\wx,x_n)=0 &    		\quad \uu(\wx,0)> 0\\ 
			  \lim_{x_n\to0^+} x_n^a\partial_n\uu(\wx,x_n)\leq 0 &    	\quad \wx\in(B_1)^+,
			\end{array} \right.
\end{align}
which are related to the study of the classical obstacle problem in $\R^{n-1}$ for fractional Laplacian $(\Delta)^s$ 
with $s\in (0,1)$, where $a=1-2s$. In particular, for all $v$ solution of $\diw(x_n^a\,\nabla v(\wx,x_n))=0$ on $B_1^+$, with 
an appropriate extension to the whole $\R^n$,  
there exists the limit $\lim_{x_n\to0^+} x_n^a\partial_n v(\wx,x_n)$  
and $\lim_{x_n\to0^+} x_n^a\partial_n v(\wx,x_n)=C(-\Delta)^s f(\wx)$ with $f$ the trace of $v$ on $\R^{n-1}\times \{0\}$ and $C$ 
a constant depending on $n$ and $s$ (cf. \cite{CS}).

For $x_n>0$, $\uu(\wx,x_n)$ is smooth so the second condition in \eqref{POF} holds in the classical sense, while the third and fourth 
condition in \eqref{POF} hold in the weak sense. 
By Silvestre \cite{S} $\uu(\wx,0)\in C^{0,\a}$ with $\a<s$, in particular if $a<\a<s$ the limit 
$\lim_{x_n\to0^+} x_n^a\partial_n\uu(\wx,x_n)$ can be considered in the classical sense. 
By \cite[Proposition 3.10]{S} we also know that $\partial_{ee}\uu\geq 0$ for all 
$e\in \mathbb{S}^{n-2}\subset \R^{n-1}\times \{0\}$, or rather $\uu$ is semiconvex in the variable $\wx$; 
moreover if the
obstacle $\phi\in C^{1,1}$ then $\partial_{ee}\uu\geq -\sup|D^2 \phi|$.

The function $\uu$, 
can be extended by simmetry $\uu(\wx,x_n)=\uu(\wx,-x_n)$. 
So, as shown in \cite{CS}
we can rewrite the problem \eqref{POF} as \eqref{POF'}.

In order to simplify the notation, we introduce the following symbol: 
\begin{equation}\label{e:def R}
 \RR(\psi):=\lim_{\varepsilon\to0^+} \varepsilon^a\partial_n\psi(\wx,\varepsilon)  
\end{equation}
for all functions $\psi$ which are solutions for
\begin{equation}\label{e:La+sym}
 \left\{  \begin{array}{ll}
			  \psi(\wx,x_n)=\psi(\wx,-x_n) \\ 
                         L_a(\psi):=\diw(|x_n|^a\,\nabla \psi(\wx,x_n))=0 & \qquad\qquad \{x_n\neq 0\}.
			  \end{array} \right.
\end{equation}

In what follows, we shall state a uniform estimate on the solution $\uu$, so we report a quantitative result stated in 
\cite[Theorem 2.1]{FS16}

\begin{teo}\label{existence regularity uu}
 For every boundary datum $g\in H^1(B_1,\mu_a)$ that respects the condition of compatibility with the problem, 
 i.e. $g(\wx,x_n)=g(\wx,-x_n)$ and $g(\wx,0)\geq 0$, there exists a unique solution $\uu$ to the fractional obstacle problem \eqref{POF'}.
 Moreover, $\partial_{x_i}\uu \in C^s(B_{1/2})$ for $i=1,\dots,n-1$ and $|x_n|^a\partial_{x_n}u\in C^{\a}(\overline{B^+_{1/2}})$
 for all $0<\a<1-s$, and
 \begin{equation}\label{estimate regularity uu}
\|u\|_{X_{s,\a}(\overline{B^+_{1/2}})}:= \|u\|_{C^0(\overline{B^+_{1/2}})}+\|\nabla_{\wx}\uu\|_{C^s(\overline{B^+_{1/2}})} + \||x_n|^a\partial_{x_n}\uu\|_{C^{\a}(\overline{B^+_{1/2}})}\leq C \|\uu\|_{L^2( B^+_1,\mu_a)},
 \end{equation}
 with $X_{s,\a}(\overline{B^+_{1/2}}):=\left\{v\in H^1(B_{1/2})\,\,:\,\,v\in C^0(\overline{B^+_{1/2}}),\, 
 \nabla_{\wx}v\in C^s(\overline{B^+_{1/2}})\,\, \textrm{and}\,\,\, |x_n|^a\partial_{x_n}v\in C^{\a}(\overline{B^+_{1/2}})\right\}$.
\end{teo}
Next, we state a
version of the Divergence Theorem that will be used frequently in the paper.
\begin{teo}[Divergence Theorem]\label{t:divergence theorem fractional}
Let $\phi\in H^1(B_1,\mu_a)$ and $\psi$ be a solution of \eqref{e:La+sym}, then
\begin{equation}\label{e:divergence theorem fractional}
\begin{split}
\int_{B_1} \nabla \psi\cdot\nabla \phi d\mu_a
= \int_{\partial B_1} \phi \,\nabla\psi \cdot x \,|x_n|^a\,d\mathcal{H}^{n-1} - 2\int_{B'_1} \phi\RR(\psi)\,d\mathcal{H}^{n-1}
\end{split}
\end{equation}
\end{teo}
We conclude the paragraph stating some results related to weighted Sobolev spaces. 
We rewrite these results for our aims, but these also hold in more general conditions. 

We state the analogous of the Banach-Alaoglu-Bourbaki Theorem (see \cite[Theorem III.15]{Brezis}) for which
every bounded and closed set in $H^1(B_1,\mu_a)$ is relatively compact in the weak topology. 
\begin{teo}[Banach-Alaoglu-Bourbaki Theorem {\cite[Theorem~1.31]{HKM}}]\label{T:1.31 HKM}
Let $v_j$ be a bounded sequence in $H^1(B_1,\mu_a)$. Then there exists a subsequence $v_{j_i}$ and a function 
 $v\in H^1(B_1,\mu_a)$ such that $v_{j_i}\rightharpoonup v$ in $L^2(B_1,\mu_a)$ 
 and $\nabla v_{j_i}\rightharpoonup \nabla v$ in $L^2(B_1,\mu_a; \R^n)$.
\end{teo}
Moreover in view of \cite[Theorem~8.1]{HK}, where Heinonen and Koskela obtained an analogous of the Rellich Theorem on Sobolev metric spaces, 
we can deduce that every bounded and closed set in $H^1(B_1,\mu_a)$ is relatively compact in $L^2(B_1,\mu_a)$.
\begin{teo}[Rellich Theorem {\cite[Theorem~8.1]{HK}}]\label{T:8.1 HK}
Let $v_j$ be a bounded sequence in $H^1(B_1,\mu_a)$. Then there exists a subsequence $v_{j_i}$ and a function 
 $v\in H^1(B_1,\mu_a)$ such that $v_{j_i}\to v$ in $L^2(B_1,\mu_a)$
\end{teo}
Futhermore, we indicate two Theorems of compact Trace embedding. We are interested in the \emph{trace} of functions in $H^1(B_1,\mu_a)$
on $L^2(B'_1)$ and $L^2(\partial B_1,|x_n|^a\,\mathcal{H}^{n-1})$.
\begin{teo}[Trace Theorem {\cite[Theorem 3.4]{Foc09}}]\label{T:tracciaFoc09}
 For all $a\in (-1, 1)$ there exists a compact operator $\tr: H^1(B^+_1,\mu_a)\to L^2(B'_1)$ 
 such that $\tr(u)=u$ for every $u\in C^\infty(\overline{B^+_1})$
\end{teo}
The Theorem of Trace embedding on $L^2(\partial B_1,|x_n|^a\,\mathcal{H}^{n-1})$
is similar to  
the Theorem of Trace embedding in the classical Sobolev spaces,
for its proof we refer to \cite[Section 3.7]{Gi03}. 
\begin{teo}[Trace Theorem]\label{T:traccia partialB1}
 For all $a\in (-1, 1)$ there exists a compact operator $\tr: H^1(B_1,\mu_a)\to L^2(\partial B_1,|x_n|^a\,\mathcal{H}^{n-1})$ 
 such that $\tr(u)=u$ for every $u\in C^\infty(\overline{B_1})$
\end{teo}

\section{Frequency formula}\label{s:3frequency}
Let $x_0\in \Gamma(\uu)$ and $r\in (0,1-|x_0|)$; let $N^{x_0}(r,\uu)$ be the frequency function defined by
\begin{equation}\label{d:frequency}
 N_a^{x_0}(r,\uu):=\frac{r\int_{B_r(x_0)}|\nabla \uu|^2\,d\mu_a}{\int_{\partial B_r(x_0)} \uu^2\,|x_n|^a\,d\mathcal{H}^{n-1}}
\end{equation}
if $\uu|_{\partial B_r(x_0)}\not\equiv 0$. We recall the monotonicity result due to Caffarelli and Silvestre \cite{CS}. 
\begin{teo}\label{frequency teo}
 \begin{itemize}
  \item[(i)] The frequency function $N_a^{x_0}(r,\uu)$ is monotone nondecreasing in the variable $r$ for all $r\in (0,1-|x_0|)$.
  \item[(ii)] For all points $x_0\in \Gamma(\uu)$ the function $N^{x_0}(r,\uu)= \l$ for all $r\in (0,1-|x_0|)$ 
  if and only if $\uu(x_0+\cdot)$ is $\l$-homogeneous. 
  \item[(iii)] If $\uu(x_0+\cdot)$ is $\l$-homogeneous then $\l\geq 1+s$.
  \item[(iv)] $N_a^{x_0}(r,\uu)\geq 1+s$ for all $x_0\in \Gamma_u$ and $r\in (0,1-|x_0|)$.
 \end{itemize} 
 \end{teo}
\begin{proof}
 As far as the proof of (i), (ii) and (iii) is concerned, we refer to \cite[Theorem 6.1]{CS} and \cite[Proposition 5.1]{CSS}.
 As regards the proof of (iv), see Remark \ref{r:proof iv teo frequency}.
\end{proof}
Thanks to Theorem~\ref{frequency teo}(i) 
it is possible to define the limit $N_a^{x_0}(0^+,u):=\lim_{r\to 0^+}N_a^{x_0}(r,u)$.
We denote by $\Gamma_{1+s}(u)$ the subset of points of free-boundary with frequency $1+s$: 
\begin{equation}\label{d:Gamma 1+s}
\Gamma_{1+s}(u):=\{x_0\in \Gamma_u \,\,:\,\,N_a^{x_0}(0^+,u)=1+s\}.
\end{equation}
Note that from the monotonicity of the frequency and by the upper semicontinuity
of the function $x\mapsto N_a^{x}(0^+,u)$\footnote{The function $N_a^{\cdot}(0^+,u)$ is the infimum on $r$ of continuous 
functions $N_a^{x}(r,u)$.}
the set $\Gamma_{1+s}\subset \Gamma_u$ is open in the relative topology.

We introduce the notation:
\begin{equation*}
 D_a^{x_0}(r)=\int_{B_r(x_0)}|\nabla \uu|^2\,d\mu_a\qquad\qquad H_a^{x_0}(r)=\int_{\partial B_r(x_0)}\uu^2\,|x_n|^a\,d\mathcal{H}^{n-1}
\end{equation*}
and we can omit to write the point $x_0$ if $x_0=\underline{0}$.

All functions $H^{x_0}_a(\cdot),D^{x_0}_a(\cdot)$ and $N^{x_0}_a(\cdot)$ are absolutely continuous functions of the radius, 
so they are differentiable a.e.

We prove two properties of $H_a^{x_0}(r)$ (see \cite[Lemma 2]{ACF}, \cite[A.2.Lemma]{FS} for the case $a=0$).
\begin{lem}\label{l:H e N}
 \begin{itemize}
  \item[(i)] The function 
  \begin{equation}
   (0,1-|x_0|)\ni r \mapsto \frac{H_a^{x_0}(r)}{r^{n+2}}
  \end{equation}
 is nondecreasing and in particular
 \begin{equation}\label{e:H(r)<}
  H_a^{x_0}(r)\leq \frac{H^{x_0}_a(1-|x_0|)}{(1-|x_0|)^{n+2}}r^{n+2} \qquad\qquad \textrm{for all}\quad 0<r<1-|x_0|.
 \end{equation}
 \item[(ii)] Let $x_0\in \Gamma_{1+s}$. For all $\varepsilon>0$ there exists an $r_0(\varepsilon)$ such that
 \begin{equation}\label{e:H(r)>r eps}
  H_a^{x_0}(r)\geq \frac{H_a^{x_0}(r_0)}{r_0^{n+2+\varepsilon}}\,r^{n+2+\varepsilon} \qquad \qquad \textrm{for all}\quad 0<r<r_0.
 \end{equation}
 \end{itemize}
\end{lem}
\begin{proof}
(i) We proceed along a two-step argument. Let $x_0\in \Gamma_{1+s}(u)$ we recall that $x_0=(\widehat{x_0},0)$.
Thanks to the Divergence Theorem and the third condition of \eqref{POF} for which $u\RR(u)=0$ in $B'_1$ 
we can compute the derivative of $\frac{H_a^{x_0}(r)}{r^{n-2s}}$:
  \begin{equation}\label{d/dr H/r^n-2s}
   \begin{split}
    \frac{d}{dr}\left(\frac{1}{r^{n-2s}}\,H_a^{x_0}(r)\right) = \frac{2}{r^{n-2s}}\,\int_{B_r(x_0)} |\nabla\uu(x)|^2\,d\mu_a.
   \end{split}
  \end{equation}
  Next, throught the equation \eqref{d/dr H/r^n-2s}, we compute the derivative of $\frac{H_a^{x_0}(r)}{r^{n+2}}$
  \begin{equation}\label{e:d/dr H/r^n+2}
   \begin{split}
    &\frac{d}{dr}\left(\frac{H_a^{x_0}(r)}{r^{n+2}}\right)
    =2\, r^{-n-3}\left(r\,\int_{B_r(x_0)} |\nabla\uu(x)|^2\,d\mu_a- (1+s)\,\int_{\partial B_r(x_0)}\uu^2\,|x_n|^a\,d\mathcal{H}^{n-1}\right),   
   \end{split}
  \end{equation}
  then, according to item (i) in Theorem \ref{frequency teo} and recalling that $x_0\in \Gamma_{1+s}(u)$ 
  we can deduce that $r^{-(n+2)}\,H_a^{x_0}(r)$ is nondecreasing.\\
  (ii) Let $r_0=r_0(\varepsilon)$ be a radius such that for all $r<r_0$ it holds $N_a^{x_0}(u)\leq (1+s)+\varepsilon/2$.
  Then, thanks to \eqref{d/dr H/r^n-2s}, we obtain
  \begin{equation*}
   N_a^{x_0}(r,u)=\frac{r}{2}\,\frac{d}{dr}\log\left(\frac{H_a^{x_0}(r)}{r^{n-2s}}\right)\leq (1+s)+\varepsilon/2.
  \end{equation*}
So, dividing to $\frac{r}{2}$ and integrating on $(r,r_0)$ we have
\begin{equation*}
  \begin{split}
    H_a^{x_0}(r)\geq H_a^{x_0}(r_0)\left(\frac{r}{r_0}\right)^{n+2+\varepsilon}. \qedhere
  \end{split}
 \end{equation*}
\end{proof}

We now prove a version 
of the Rellich formula for weighted Sobolev spaces: 
\begin{prop}[Rellich formula]\label{Rellich formula}
 Let $v$ be a solution of \eqref{POF'}. 
 Then it holds that:
 \begin{equation*}
  \int_{\partial B_r} |\nabla v|^2\,|x_n|^a\,d\mathcal{H}^{n-1}=\frac{n-2+a}{r}\int_{B_r}|\nabla v|^2\,d\mu_a + 2\,\int_{\partial B_r} \left(\langle\nabla v,\frac{x}{r}\rangle\right)^2\,|x_n|^a\,d\mathcal{H}^{n-1}.
 \end{equation*}
\end{prop}
\begin{proof}We apply the Divergence Theorem and the third condition of \eqref{POF} for which $u\RR(u)=0$ in $B'_1$ and develop 
\begin{equation*}
 \diw\left(|\nabla v|^2\frac{x}{r}\,|x_n|^a -2\,\langle\nabla v,\frac{x}{r}\rangle \nabla v\,|x_n|^a\right). \qedhere
\end{equation*}
\end{proof}
In view of section~\ref{s:epiperimetric and consequences} we compute the derivative of the volume and boundary energies.
\begin{lem}\label{l:H' e D'} 
The following formulae hold:
 \begin{itemize}
  \item[(i)] $(H^{x_0}_a)'(r)=\frac{n-2s}{r}H_a^{x_0}(r) + 2\int_{\partial B_r(x_0)} u\nabla u \cdot\nu\,|x_n|^a\,d\mathcal{H}^{n-1}$;
  \item[(ii)] $(D^{x_0}_a)'(r)=\frac{n-2+a}{r}D_a^{x_0}(r) + 2\int_{\partial B_r(x_0)} (\nabla u \cdot\nu)^2\,|x_n|^a\,d\mathcal{H}^{n-1}$;
  \item[(iii)] $D^{x_0}_a(r)=\int_{\partial B_r(x_0)} u\nabla u \cdot\nu\,|x_n|^a\,d\mathcal{H}^{n-1}$;
 \end{itemize}
\end{lem}
\begin{proof}
 (i) We can obtain the thesis observing that  $\frac{d}{dr}u^2(x_0 + ry)=2\,u(x_0 + ry)\nabla u(x_0 + ry)\cdot y$.\\
(ii) From Coarea and Rellich Formulae we obtain 
\begin{equation*}
\begin{split}
 (D^{x_0}_a)'(r)=
 \int_{\partial B_r(x_0)} |\nabla u|^2\,d\mu_a 
 \stackrel{Prop. \ref{Rellich formula}}{=} &\,\,\frac{n-2+a}{r}\,D^{x_0}_a(r) + 2\int_{\partial B_r(x_0)} (\nabla u \cdot\nu)^2\,|x_n|^a\,d\mathcal{H}^{n-1}.
\end{split}
\end{equation*}
(iii) In order to prove the formula, it is enough to apply the the Divergence Theorem and the third condition of \eqref{POF} for which $u\RR(u)=0$ in $B'_1$.
\end{proof}
\section{The blow-up method: existence and $(1+s)$-homogeneity of blow-ups}\label{s:3existence}
In order to study the properties of the free-boundary,  we investigate the properties 
of the blow-up limits. 
We shall consider a suitable sequence of rescaled functions of the solution $\uu$. 
Let $x_0\in \Gamma_{1+s}(u)$, we set
\begin{equation}\label{d:ur 1+s}
 u_{x_0, r}(x):=\frac{u(x_0+rx)}{r^{1+s}},
\end{equation}
if $x_0=\underline{0}$ we denote $u_r(x)$ in the place of $u_{\underline{0}, r}(x)$.
Note that in the choice of the rescaling factor in \eqref{d:ur 1+s} we follow the same approach as in \cite{FS} and \cite{GPPS},
which is different with respect to the previous approach used in \cite{ACS}. 

The first step in the analysis of blow-ups is to prove 
the existence of the limits of the sequence $(u_{x_0, r})_r$ for all $x_0\in \Gamma_{1+s}(u)$. 
In order to prove their existence, we state the equiboundedness 
of $(u_{x_0, r})_r$ with respect to the $H^1(B_1,\mu_a)$-norm. 
\begin{prop}[Existence of blow-ups]
 Let $\uu\in H^1(B_1,\mu_a)$ be the solution of \eqref{POF'} and let $x_0\in \Gamma_{1+s}(u)$. 
Then for every sequence $r_k\downarrow 0$ there exists a subsequence $(r_{k_j})_j\subset (r_k)_k$ such that the rescaled functions 
$(u_{x_0,r_{k_j}})_j$ converge in $L^2(B_{1-|x_0|},\mu_a)$.
\end{prop}
\begin{proof}
Since $x_0\in \Gamma_{1+s}(u)$,
 \begin{equation}\label{e:estimate nabla ur}
 \begin{split}
  \|\nabla u_{x_0,r_k}\|^2_{L^2(B_1,|x_n|^a;\R^n)}=\frac{D_a^{x_0}(r_k)}{r_k^{n+1}}
  \stackrel{\eqref{d:frequency}}{=}\frac{H_a^{x_0}(r_k)}{N_a^{x_0}(r_k)\,r_k^{n+2}}
  \leq \frac{H_a^{x_0}(1-|x_0|)}{(1+s)\,(1-|x_0|)^{n+2}}
 \end{split}
 \end{equation}
 where in the last inequality we used the inequality \eqref{e:H(r)<} and the Theorem \ref{frequency teo}(i).
Due to Lemma \ref{l:H e N}(i), we have 
 \begin{equation*}
  \|u_{x_0,r_k}\|^2_{L^2(\partial B_1,\mu_a)}=\frac{H_a^{x_0}(r_k)}{r_k^{n+2}}
  \stackrel{\eqref{e:H(r)<}}{\leq} \frac{H_a^{x_0}(1-|x_0|)}{(1-|x_0|)^{n+2}}. 
 \end{equation*}
So, according to the 
Poincaré inequality we have 
\begin{equation}\label{sup ur L2 lim}
\begin{split}
 \sup_k \|u_{x_0,r_k}\|_{L^2(B_1,\mu_a)}
 \leq C \left(\sup_k \|u_{x_0,r_k}\|_{L^2(\partial B_1,|x_n|^a\,\mathcal{H}^{n-1})}+\sup_k \|\nabla u_{x_0,r_k}\|_{L^2(B_1,\mu_a;\R^n)}\right)<\infty. 
 \end{split}
\end{equation}
 Therefore, thanks to Theorem~\ref{T:1.31 HKM} 
 for every subsequence  of radii $r_k\searrow 0$, there exists an extracted
 subsequence $r_{k_j}\searrow 0$ such that  $u_{x_0, r_{k_j}} \to u_0$ in $L^2(B_{1-|x_0|},\mu_a)$ as $j\to +\infty$.
\end{proof}

\begin{oss}
So, according to the quantitative estimate \eqref{estimate regularity uu} and inequality \eqref{sup ur L2 lim} 
\begin{equation}\label{u c1+s lim}
\begin{split}
 \sup_k \|u_{x_0,r_k}&\|_{X_{s,\a}(\overline{B_{1/2}})}
 \stackrel{\eqref{estimate regularity uu}}{\leq} \sup_k \|u_{x_0,r_k}\|_{L^2(B_1,\mu_a)}<\infty. 
 \end{split}
\end{equation}
In particular, in view of \eqref{u c1+s lim}
 we can easily deduce that
 \begin{equation}\label{u e grad u limitate 1+s}
  \|u\|_{L^\infty(B_r(x_0))}\leq C r^{1+s},\qquad \|\nabla_{\wx} u\|_{L^\infty(B_r(x_0);\R^n)}\leq C r^{s}
  \quad\mathrm{and}\quad \||x_n|^a\partial_{x_n}u\|_{L^\infty(B_r(x_0);\R^n)}\leq C r^{1-s}.
  \end{equation}
\end{oss}

Similarly to \cite{Weiss}  
we consider an energy ``à la Weiss''  
used in \cite{FS} and \cite{GP09} for fractional Laplacian 
(see \cite{GPPS} for a version in the fractional Laplacian problem with drift and 
\cite{FGS,Geraci16} for a version in the classical obstacle problem with quadratic energies with variable coefficients):
\begin{equation}\label{e:Wiess monotonicity frac}
W^{x_0}_{1+s}(r,u)=\frac{1}{r^n+1}\int_{B_r(x_0)}|\nabla u|^2 \,|x_n|^a\, dx - \frac{1+s}{r^{n+2}}\int_{\partial B_r(x_0)}u^2\,|x_n|^a\, d\mathcal{H}^{n-1}.
\end{equation}
We note that 
\begin{equation*}
 W^{x_0}_{1+s}(r,u)= \frac{H^{x_0}_{a}(r)}{r^{n+2}}(N_a^{x_0}(r,u)-(1+s)),
\end{equation*}
thus if $x_0\in \Gamma_{1+s}(u)$ by \eqref{d:Gamma 1+s} and Lemma \ref{l:H e N} 
(which guarantees the boundedness of $\frac{H^{x_0}_{a}(r)}{r^{n+2}}$)
we have
\begin{equation*}
 \lim_{r\searrow 0} W^{x_0}_{1+s}(r,u)=0
\end{equation*}
and due to Theorem \ref{frequency teo}, we obtain 
\begin{equation*}
 W^{x_0}_{1+s}(r,u)\geq 0.
\end{equation*}
Moreover, the function $W^{x_0}_{1+s}(\cdot,u)$ satisfies a monotonicity formula in the same essence as Weiss' monotonicity 
formula proved in \cite{Weiss}. 
For a similar proof see \cite[Theorem 3.5]{GPPS}. 
\begin{prop}[Weiss' monotonicity formula]\label{p:Weiss 1+s}
Let $x_0\in \Gamma_{1+s}(x_0)$ and $\uu$ be a solution of Problem \eqref{POF}; 
then the function $r\mapsto W^{x_0}_{1+s}(r,u)$ is nondecreasing. 
In particular, the following formula holds:
 \begin{equation*}
  \frac{d}{dr}W^{x_0}_{1+s}(r,u) = \frac{2}{r} \int_{\partial B_1} \left( \nabla u_r\cdot \nu - (1+s)u_r\right)^2\,|x_n|^a\,d\mathcal{H}^{n-1}
 \end{equation*}
\end{prop}

 Next, we prove the homogeneity property of blow-ups. 
 We prove the result through properties of the frequency function and 
 the optimal regularity of the solution.
 Proceeding as in \cite[Proposition 4.2]{FGS} and thanks to Proposition~\ref{p:Weiss 1+s} 
 it is possible to obtain the same result.
\begin{prop}[$(1+s)$-homogeneity of blow-ups]\label{p:1+s homogeneity}
 Let $\uu\in H^1(B_1, \mu_a)$ be a solution of Problem \eqref{POF'}.
 Let $x_0\in \Gamma_{1+s}(u)$ and $(u_{x_0,r})_r$ be a sequence of rescaled functions. 
 Then, for every sequence $(r_j)_j\downarrow 0$ there exists
 a subsequence $(r_{j_k})_k\subset(r_j)_j$ such that the sequence $(u_{x_0,r_{j_k}})_k$ converges in $C^{1+\a}(\R^n)$ (see \eqref{estimate regularity uu}) 
 for all $\a<s$ to $u_{x_0}$ a $(1+s)$-homogeneous function.
\end{prop}
\begin{proof}
In view of \eqref{u c1+s lim}, thanks to the Ascoli-Arzelà Theorem there exists a subsequence (that we do not relabel)
$\uu_{x_0,r_k}$ and $u_{x_0}\in X_{s,\a}(\overline{B_{1/2}})$ such that $\|u_{x_0,r_k}-u_{x_0}\|_{X_{\beta,\a}(\overline{B_{1/2}})}$ converge to $0$ 
for all $\beta<s$.
It is easy to prove that $u_{x_0}$ is a solution of Problem \eqref{POF'}.
 In order to conclude the proof, we show that $u_{x_0}$ is $(1+s)$-homogeneous.
 
 We note that for every $\d>0$ we can fix $\rho>0$ such that $N_a^{x_0}(\rho,u)\leq (1+s)+\d$. 
 So for $k>>1$, for every $t\in (0,1)$
 (such that $t\,r_k<\rho$)
 \begin{equation}\label{e:conto hom}
  \begin{split}
   N_a(t,\uu_{x_0,r_k})=&N_a^{x_0}(t,u_{r_k})=N_a^{x_0}(t\,r_k,u)-N_a^{x_0}(\rho,u) +N_a^{x_0}(\rho,u) \leq (1+s) +\d,\\
   &N_a^{x_0}(t,u_{r_k})=N_a^{x_0}(t\,r_k,u)\geq 1+s
  \end{split}
 \end{equation}
where we resort to Theorem \ref{frequency teo}. 
Now, from the convergence of $\uu_{x_0,r_k}$ to $u_{x_0}$ and thanks to the arbitrariness of $\d$, we obtain $N_a(t,u_{x_0})\equiv 1+s$;
then, by Theorem \ref{frequency teo}(ii), $u_{x_0}$ is $(1+s)$-homogeneous. 
\end{proof}
\begin{oss}\label{r:proof iv teo frequency}
 By proceeding in the same way, we can prove Theorem \ref{frequency teo}(iv) as well:
\end{oss}
\begin{proof}[Proof of Theorem \ref{frequency teo}(iv)]
 Let $x_0\in\Gamma(u)$ and $\l=N_a^{x_0}(0^+,u)$. Then, if $r_k\searrow 0$ is a suitable sequence of radii, for all $\d>0$ we can fix $\rho>0$ such that $N_a^{x_0}(\rho,u)\leq \l+\d$. 
 So, proceeding in much the same way as in \eqref{e:conto hom}, we deduce
 \begin{equation*}
  \begin{split}
   \l\leq N_a(t,u_{x_0},r_k) \leq \l + \d,
  \end{split}
 \end{equation*}
 thus, by the strong convergence of $u_{x_0,r_k,}$ to its blow-up $w_0$ and by the arbitrariness of $\d$, we have
 $N_a(t,w_0)\equiv \l.$ So, by the second item of Theorem \ref{frequency teo} $w_0$ is $\l$-homogeneous and by Theorem \ref{frequency teo}(iii) 
 $\l\geq 1+s$.
\end{proof}

\section{Classification of the $(1+s)$-homogeneous global solutions}\label{s:3classification}
Let $h_e$ be the function defined by
\begin{equation}\label{he}
 h_e(x):=\left(s^{-1}\wx\cdot e -\sqrt{(\wx\cdot e)^2+x_n^2}\right)\left(\sqrt{(\wx\cdot e)^2+x_n^2}+\wx\cdot e\right)^s.
\end{equation}
From a simple calculation it is possible to prove the following properties:
\begin{itemize}
 \item[(i)] $h_e(\wx,x_n)=h_e(\wx,-x_n)$;
 \item[(ii)] $h_e(x)\geq 0$ on $\{x_n=0\}$ and $h_e=0$ on $\{x_n=0,\, \wx\cdot e\leq 0\}$;
 \item[(iii)] $\partial_e h_e(x) = \frac{1-s^2}{s}\left(\sqrt{(\wx\cdot e)^2+x_n^2}+\wx\cdot e\right)^s$; 
 \item[(iv)] $\partial_n h_e(x) = -(1+s)x_n\left(\sqrt{(\wx\cdot e)^2+x_n^2}+\wx\cdot e\right)^{s-1}$;
 \item[(v)] $h_e$ is solution of \eqref{e:La+sym};
\item[(vi)]
\begin{equation}\label{lim_eh}
  \RR h_e(\wx)=\left\{  \begin{array}{ll}
			  0 & 							\quad \wx\cdot e\geq 0\\ 
			  -(1+s)\left(2\,|\wx\cdot e|\right)^{1-s} &     	\quad \wx\cdot e< 0.
			\end{array} \right.
\end{equation}
In particular, we obtain a complementarity property
\begin{equation}\label{complementarity}
 h_e(\wx,x_n)\, \RR h_e(\wx)=0 \qquad\qquad \mathrm{on}\,\, \{x_n=0\}
\end{equation}
\end{itemize}

In view of properties above, $h_e$ is a solution of problem \eqref{POF}, so by \cite{S} $\partial_{\tau\tau}h_e\geq0$ for any vector 
$\tau\in \mathbb{S}^n\subset\R^{n-1}\times \{0\}$. So, thanks to its $(1+s)$-homogeneity, 
$h_e$ is a solution of 
\begin{equation}\label{e:global solution}
 \left\{  \begin{array}{ll}
			  v(\wx,0)\geq 0 & 						\quad \wx\in \R^{n-1}\\ 
			  v(\wx,x_n)=v(\wx,-x_n) &	\\ 
			  \diw(|x_n|^a\,\nabla v(\wx,x_n))=0 &			\quad x\in \R^n\setminus\{(\wx,0)\,:\,\uu(\wx,0)=0\}\\ 
			  \diw(|x_n|^a\,\nabla v(\wx,x_n))\leq 0 &    		\quad x\in \R^n\,\,\textrm{in distributional sense}\\
			  \partial_{\tau\tau}v\geq 0 & 				\quad \textrm{for any vector $\tau\in \partial B'_1$.} 
			\end{array} \right.
\end{equation}

According to \cite[Proposition 5.5]{CSS}, the function $h_e$ is, up to a rotation and the product by scalar, 
the unique $(1+s)$-homogeneous, global solution of \eqref{e:global solution}.

We consider the closed convex cone of $(1+s)$-homogeneous global solutions :
\begin{equation}\label{e:H 1+s set}
 \mathfrak{H}_{1+s}:=\{\l\,h_e\,\,:\,\, e\in \mathbb{S}^{n-2}, \l\in [0,+\infty)\}\subset H^1_{loc}(\R^n,\mu_a).
\end{equation}
Caffarelli, Salsa and Sivestre \cite{CSS} proved that $\mathfrak{H}_{1+s}\setminus\{\underline{0}\}$ is the set of blow-ups 
in the regular points of the free-boundary with lower frequency.\\
We note that $\mathfrak{H}_{1+s}$ is a closed cone in $H^1_{loc}(\R^n,\mu_a)$. The restriction
\begin{equation*}
 \mathfrak{H}_{1+s}|_{B_1}:=\{v|_{B_1}\,\,:\,\, v\in\mathfrak{H}_{1+s}\}\subset H^1(B_1,\mu_a)
\end{equation*}
is a closed set, and $\mathfrak{H}_{1+s}\setminus\{0\}$ is parameterized by a $(n-1)$-manifold by the map
\begin{align*}
 \mathbb{S}^{n-2}&\times (0,\infty)\xrightarrow{\Phi} \mathfrak{H}_{1+s}\setminus\{0\}\\
 (e&,\l)\quad\quad\longmapsto \quad\l\,h_e.
\end{align*}
Next we can introduce the tangent plane to space $\mathfrak{H}_{1+s}$ in every point $\l\,h_e$ as
\begin{equation}\label{def piano tangente}
 T_{\l\,h_e}\mathfrak{H}_{1+s}:=\{d_{(e,\l)}\Phi(\xi,\a)\,\,:\,\,\xi\cdot e_n=\xi\cdot e=0, \a\in \R\}
\end{equation}
We compute the derivative of the map $\Phi$ in a point of $\mathbb{S}^{n-2}\times (0,\infty)$:  
\begin{align}\label{diffPHI}
 d_{(e,\l)}\Phi(\xi,\a)= \frac{d}{dt}h_{\sigma(t)}|_{t=0}
\end{align}
with $\sigma(t)=\frac{e+t\xi}{\|e+t\xi\|}$, a curve on $\mathbb{S}^{n-2}$ such that $\sigma(0)=e$ and $\sigma'(0)=\xi$. 
By \eqref{he} and \eqref{diffPHI} we obtain
\begin{align*}
 {\frac{d}{dt}h_{\sigma(t)}}_{|t=0}
 =(s^{-1}-s)\,\wx\cdot\xi\left(\sqrt{(\wx\cdot e)^2+x_n^2}+\wx\cdot e\right)^s.
\end{align*}
Then, we can rewrite \eqref{def piano tangente} as
\begin{equation*}
 T_{\l\,h_e}\mathfrak{H}_{1+s}:=\{\a h_e + v_{e,\xi}\,\,:\,\,\xi\cdot e_n=\xi\cdot e=0, \a\in \R\}
\end{equation*}
where the function $v_{e,\xi}$ is defined as follows:
\begin{equation*}
v_{e,\xi}=\wx\cdot\xi\left(\sqrt{(\wx\cdot e)^2+x_n^2}+\wx\cdot e\right)^s. 
\end{equation*}
We highlight some properties of function $\psi\in \mathfrak{H}_{1+s}$. 
For all $\phi\in H^1(B_1,\mu_a)$, integrating by parts, according to Theorem~\ref{t:divergence theorem fractional} and
Euler's homogeneous function Theorem we obtain
\begin{equation}\label{diw+Eul}
 \begin{split}
\int_{B_1} \nabla \psi\cdot\nabla \phi d\mu_a
&= \int_{\partial B_1} \phi \,\nabla\psi \cdot x \,|x_n|^a\,d\mathcal{H}^{n-1} - 2\int_{B'_1} \phi\RR(\psi)\,d\mathcal{H}^{n-1}\\
&=(1+s)\int_{\partial B_1} \phi \psi\, |x_n|^a\,d\mathcal{H}^{n-1} - 2\int_{B'_1} \phi\RR(\psi)\,d\mathcal{H}^{n-1}.
\end{split}
\end{equation}

\begin{oss}
 The first variation of functional $W^{\underline{0}}_{1+s}(1,\cdot)$ in a point $\psi\in\mathfrak{H}_{1+s}$ 
 along a direction $\phi\in H^1(B_1,\mu_a)$ is\footnote{The first variation is defined as $\delta W^{\underline{0}}_{1+s}(1,\psi)[\phi]
  :=\lim_{t\to 0}\left(\frac{W^{\underline{0}}_{1+s}(1,\psi+t\phi)-W^{\underline{0}}_{1+s}(1,\psi)}{t}\right)$}
 \begin{align*}
  \delta W^{\underline{0}}_{1+s}(1,\psi)[\phi]
  =2\int_{B_1}\nabla\psi\cdot\nabla\phi\,d\mu_a-2(1+s)\int_{\partial B_1}\psi\,\phi\,|x_n|^a\,d\mathcal{H}^{n-1}.
 \end{align*}
  Then, by \eqref{diw+Eul} 
\begin{equation}\label{dGG}
 \delta W^{\underline{0}}_{1+s}(1,\psi)[\phi]= -4\int_{B'_1}\phi\RR(\psi)(\wx) \,d\mathcal{H}^{n-1},
\end{equation}
by \eqref{complementarity}
\begin{equation}
 \delta W^{\underline{0}}_{1+s}(1,\psi)[\psi]=0,
\end{equation}
so we can infer that
\begin{equation}\label{GG=dGG}
 W^{\underline{0}}_{1+s}(1,\psi)=\frac{1}{2} \delta W^{\underline{0}}_{1+s}(1,\psi)[\psi]=0 \qquad\qquad \forall\psi\in\mathfrak{H}_{1+s}.
\end{equation}
\end{oss}

\section{The epiperimetric inequality and its consequences}\label{s:epiperimetric and consequences}
In this section we prove an epiperimetric inequality for the points in $\Gamma_{1+s}(u)$, and its main consequences 
in the framework of the regularity of the free-boundary.
In Paragraph~\ref{s:proof epiperimetric} we prove the epiperimetric inequality. In Paragraph~\ref{s:decay} we establish
a decay estimate for adjusted boundary energy. In Paragraphs~\ref{s:3nondegeneration} and \ref{s:3!blowup} we state
the nondegeneracy of the solution and the uniqueness of the blow-ups in $\Gamma_{1+s}(u)$ respectively.

\subsection{Epiperimetric inequality}\label{s:proof epiperimetric}
We now state the main result of this paper: 
the \emph{epiperimetric inequality} ``à la Weiss'' in our setting.  
This result is a key ingredient in our approach to the decay of the boundary adjusted energy and to the uniqueness of blow-ups 
(see \cite{FS} for the classical case of Laplacian $s=1/2$).

In this paragraph we state and prove the epiperimetric inequality. For the convenience of readers, 
the proof will be split into several steps.

\begin{teo}[Epiperimetric inequality]\label{t:epiperimetric inequality}
 There exists a dimensional constant $\kappa\in(0,1)$ such that if $c\in H^1(B_1,\mu_a)$ is a $(1+s)$-homogeneous function 
 with $c\geq 0$ on $B'_1$ and $c(\wx,x_n)=c(\wx,-x_n)$
then
\begin{equation}\label{e:epip}
 \inf_{v\in \mathfrak{A}_c} W^{\underline{0}}_{1+s}(v)\leq (1-\kappa)W^{\underline{0}}_{1+s}(c).
\end{equation}
\end{teo}

\begin{proof}
Without loss of generality it is possible to suppose that the function $c$ satisfies the follows condition
\begin{equation}\label{e:epiperimetric cond delta}
\mathrm{dist}_{H^1(B_1,\mu_a)}(c,\mathfrak{H}_{1+s})<\d.
 \end{equation}
In fact, according to the $(1+s)$-homogeneity of $c$ and recalling that $\mathfrak{H}_{1+s}$ is a cone, 
for all $\d>0$ there exists a constant $\gamma>0$ such that
\begin{equation*}
 \mathrm{dist}_{H^1(B_1,\mu_a)}(\gamma c,\mathfrak{H}_{1+s})<\d.
\end{equation*}
We can observe that  if $v\in \mathfrak{A}_{\gamma c}$ then $\gamma^{-1}v\in \mathfrak{A}_{c}$.
So, if we prove inequality \eqref{e:epip} for the function $\c c$, or rather  
\begin{equation*}
 \inf_{v\in \mathfrak{A}_{\gamma c}} W^{\underline{0}}_{1+s}(1,v)\leq (1-\kappa)W^{\underline{0}}_{1+s}(1,\gamma c),
\end{equation*}
then, thanks to $W^{\underline{0}}_{1+s}(1,\gamma c)=\gamma^2 W^{\underline{0}}_{1+s}(1,c)$ we infer
\begin{equation*}
 \inf_{w\in \mathfrak{A}_{c}} W^{\underline{0}}_{1+s}(1,w)\leq (1-\kappa)W^{\underline{0}}_{1+s}(1,c).
\end{equation*}
To simplify the notation we denote the functional $W^{\underline{0}}_{1+s}(1,\cdot)$ by $\GG(\cdot)$.\\
 We argue by contradiction. Let us suppose the existence of sequences of positive numbers $\kappa_j,\d_j\downarrow 0$ and a sequence of 
 $(1+s)$-homogeneous functions $c_j\in H^1(B_1,\mu_a)$ with $c_j\geq 0$ on $B'_1$ such that
 \begin{align}\label{1abs}
  \mathrm{dist}_{H^1(B_1,\mu_a)}(c_j,\mathfrak{H}_{1+s})=\d_j,	\\\label{2abs}
  (1-\kappa)\GG(c) \leq \inf_{v\in \mathfrak{A}_c} \GG(v) .
 \end{align}
In particular, fixing $h:=h_{e_n}$, up to change of coordinate depending on j, we assume that there exists $\l_j\geq 0$ for which $\psi_j:=\l_jh$
is the point satisfying the minimum distance between $c_j$ and $\mathfrak{H}_{1+s}$, or rather
\begin{equation}\label{psi_j-c_j=dist(c_j,H_1+s)}
 \|\psi_j-c_j\|_{H^1(B_1,\mu_a)} = \mathrm{dist}_{H^1(B_1,\mu_a)}(c_j,\mathfrak{H}_{1+s})=\d_j, \qquad \forall j\in\N.
\end{equation}
We split the proof into some intermediate steps.\\

\textbf{Step 1: Auxiliary functionals.} We can rewrite \eqref{2abs} and interpret this inequality as a condition of quasi-minimality for 
a sequence of new functionals. Setting $j\in\N$, let $v\in\mathfrak{A}_{c_j}$, we use \eqref{dGG} (applied twice to $\psi_j$
with test functions $c_j-\psi_j$ and $v-\psi_j$) and \eqref{GG=dGG}; we can rewrite \eqref{2abs}:
\begin{equation}\label{e:rewrite1 2abs}
 \begin{split}
 (1-\kappa_j)&\left(\GG(c_j)-\GG(\psi_j)-\d\GG(\psi_j)[c_j-\psi_j]-4\int_{B'_1}(c_j-\psi_j)\RR(\psi_j)\, d\mathcal{H}^{n-1}\right)\\
 &\leq \GG(v)-\GG(\psi_j)-\d\GG(\psi_j)[v-\psi_j] - 4\int_{B'_1}(v-\psi_j)\RR(\psi_j)\, d\mathcal{H}^{n-1}.
 \end{split}
\end{equation}
We can observe that $\GG(v_1)-\GG(v_2)-\d\GG(v_2)[v_1-v_2]=\GG(v_1-v_2)$, then for all $v\in\mathfrak{A}_{c_j}$ 
\eqref{e:rewrite1 2abs} can be rewritten as
\begin{equation}\label{q-m GG(psi_j)}
 \begin{split}
  (1-\kappa_j)&\left(\GG(c_j-\psi_j) - 4\int_{B'_1}(c_j-\psi_j) \RR(\psi_j)\, d\mathcal{H}^{n-1}\right)  \\
&\phantom{AAAAAAAAAAAAAAA}\leq \GG(v-\psi_j) - 4\int_{B'_1}(v-\psi_j)\RR(\psi_j)\, d\mathcal{H}^{n-1}.
 \end{split}
\end{equation}
Next we define new sequences of functions 
\begin{equation}\label{def z_j}
z_j:=\frac{c_j-\psi_j}{\d_j}  
\end{equation}
(recalling that $\psi_j=\l_jh$), positive numbers $\theta_j:=\frac{\l_j}{\d_j}$ and sets $\mathcal{B}_j:=\{z\in z_j+H^1_0(B_1,\mu_a)\,\,:\,\, (z+\theta_j h)|_{B'_1}\geq 0\}$.
Now we introduce a sequence of auxiliary functionals $\GG_j: L^2(B_1,\mu_a)\to (-\infty,+\infty]$
\begin{equation}\label{GG_j}
      \GG_j(z):=\left\{  \begin{array}{ll}
	  \displaystyle \int_{B_1}|\nabla z|^2\,d\mu_a - (1+s)\int_{\partial B_1} z_j^2\,|x_n|^a\,d\mathcal{H}^{n-1} - 4\theta_j\int_{B'_1}z\RR(h)\,d\mathcal{H}^{n-1}  &\\
	  & \!\!\!\!\!\!\!\!\!\!\!\!\!\!\!\!\!\!\!\!\!\!		\textit{if }\, z\in \mathcal{B}_j\\
	  +\infty & \!\!\!\!\!\!\!\!\!\!\!\!\!\!\!\!\!\!\!\!\!\!				\textit{otherwise}.
			\end{array} \right.
\end{equation}
We can observe that the second term in the formula above does not depend on $z$ but only on its boundary 
datum $z|_{\partial B_1}=z_j|_{\partial B_1}$.

We can rewrite \eqref{q-m GG(psi_j)} with the new notation and obtain
\begin{equation*}
 \begin{split}
  (1-\kappa_j)&\left(\GG(\d_jz_j) - 4\d_j\int_{B'_1}z_j \RR(\l_j h)\, d\mathcal{H}^{n-1}\right)  
\leq \GG(\d_j z) - 4\d_j\int_{B'_1}z\RR(\l_j h)\, d\mathcal{H}^{n-1}
 \end{split} 
\end{equation*}
and dividing by $\d_j^2$ we obtain the condition of quasi-minimality for $z_j$ with respect to $\GG_j$:
\begin{equation}\label{quasiminimality z_j GG}
 (1-\kappa_j)\GG_j(z_j)\leq \GG_j(z) \qquad\qquad \forall z\in L^2(B_1,\mu_a).
\end{equation}

Therefore we note that by the very definitions of $z_j$ and $\d_j$ we have 
\begin{equation}\label{|z_j|=1}
\|z_j\|_{H^1(B_1,\mu_a)}=1.
\end{equation}
So, by the compactness of Sobolev embedding from $H^1(B_1,\mu_a)$ into the space $L^2(B_1, \mu_a)$ Theorem~\ref{T:1.31 HKM}, 
the trace operator from $H^1(B_1,\mu_a)$ into the space $L^2(B'_1)$ Theorem~\ref{T:tracciaFoc09}, 
and the trace operator from $H^1(B_1,\mu_a)$ into $L^2(\partial B_1, |x_n|^a\mathcal{H}^{n-1})$ Theorem~\ref{T:traccia partialB1},
we may extract a subsequence (which we do not relabel) such that
\begin{itemize}
 \item[(a)] $(z_j)_{j\in \N}$ converges weakly in $H^1(B_1,\mu_a)$ to some $z_\infty$;
 \item[(b)] the sequences of traces $z_j|_{B'_1}$ and $z_j|_{\partial B_1}$ converge respectively 
	    in $L^2(B'_1)$ and $L^2(\partial B_1,|x_n|^a\mathcal{H}^{n-1})$;
 \item[(c)] $\theta_j$ has a limit $\theta\in [0,\infty]$.
\end{itemize}

\textbf{Step 2: First property of $(\GG_j)_{j\in\N}$.} In this step we establish the equi-coercivity and some other properties of
the family $(\GG_j)_{j\in \N}$.

We observe that for all $w\in\mathcal{B}_j$, since $w|_{\partial B_1}=z_j|_{\partial B_1}$ and $h\RR(h)(\wx)=0$, it holds that
\begin{equation}\label{-wdh>0}
\begin{split}
 -\int_{B'_1}\!\!w\RR(h)(\wx)\,d\mathcal{H}^{n-1}
 = -\int_{B'_1}\!\!(w+\theta_j h)\RR(h)(\wx)\,d\mathcal{H}^{n-1}+\theta_j\int_{B'_1}\!\!h\RR(h)(\wx)\,d\mathcal{H}^{n-1}\geq 0
\end{split}
\end{equation}
where we used \eqref{lim_eh} for which $\RR(h)(\wx)\leq 0$ and the condition $w\in \mathcal{B}_j$ for which $(w+\theta_j h)_{|B_1'}\geq 0$. 
Then from the definition of \eqref{GG_j} we have
\begin{equation}\label{cond_to_equicoerc_GGj}
 \int_{B_1}|\nabla w|^2\,d\mu_a - (1+s)\int_{\partial B_1}z_j^2\,|x_n|^a\,d\mathcal{H}^{n-1}\leq \GG_j(w).
\end{equation}
This establishes the equi-coercivity of the sequence $\GG_j$, in fact from \eqref{|z_j|=1}, thanks to strong convergence of traces, 
we obtain
\begin{equation*}
 \liminf_{j\in \N}\GG_j(z_j)\geq -(1+s)\int_{\partial B_1}z_\infty^2\, |x_n|^a\, d\mathcal{H}^{n-1} - 4\theta\int_{B'_1}z_\infty \RR(h)\,d\mathcal{H}^{n-1};
\end{equation*}
while if $\theta=+\infty$ from \eqref{|z_j|=1} and \eqref{cond_to_equicoerc_GGj} we conclude that
\begin{equation*}
 \liminf_{j\in \N}\GG_j(z_j)\geq -(1+s)\int_{\partial B_1}z_\infty^2\, |x_n|^a\,d\mathcal{H}^{n-1}.
\end{equation*}
Note that it is not restrictive (up to subsequence) to assume that $\GG_j(z_j)$ has a limit in $(-\infty,+\infty]$. 
Finally we can observe that 
\begin{equation}\label{lim GGj=+infty}
 \lim_{j\to \infty} \GG_j(z_j)=+\infty \qquad\Longleftrightarrow\qquad \lim_{j\to \infty} \theta_j \int_{B'_1}z_j \RR(h)\,d\mathcal{H}^{n-1}=-\infty.
\end{equation}
\textbf{Step 3: Asymptotic analysis of $(\GG_j)_{j\in\N}$.} 
In this step we prove a result of $\Gamma$-convergence for the family of functionals $(\GG_j)_{j\in\N}$.

We can distinguish three cases:\\
 (1) If $\theta\in[0,+\infty)$, then $(z_\infty+\theta h)|_{B'_1}\geq 0$ and
 $\Gamma(L^2(B_1,\mu_a))$-$\lim\GG_j=\GG_\infty^{(1)}$ with
 \begin{equation*}
	  \GG_\infty^{(1)}(z):=\left\{  \begin{array}{ll}
	  \displaystyle \int_{B_1}|\nabla z|^2\,d\mu_a - (1+s)\int_{\partial B_1} z_\infty^2\,|x_n|^a\,d\mathcal{H}^{n-1} - 4\theta\int_{B'_1}z\RR(h)\,d\mathcal{H}^{n-1} 	&\\
	  &\!\!\!\!\!\!\!\!\!\!\!\!\!\!\!\!\!\!\!\!\!\!\!\!\! 						\textrm{if}\, z\in \mathcal{B}^{(1)}_\infty\\
	  +\infty &\!\!\!\!\!\!\!\!\!\!\!\!\!\!\!\!\!\!\!\!\!\!\!\!\!   					\textrm{otherwise},
			\end{array} \right.
 \end{equation*}
where $\mathcal{B}^{(1)}_\infty:=\{z\in z_\infty+ H^1_0(B_1,\mu_a)\,\,:\,\, (z+\theta h) |_{B'_1}\geq 0\}$.\\
 (2) If $\theta=+\infty$ and $\lim_j \GG_j(z_j)<\infty$, then $z_\infty|_{B_1^{',-}}=0$ (where $B^{',-}_1=B'\cap \{x_{n-1}\leq 0\}$)
and $\Gamma(L^2(B_1,\mu_a))$-$\lim\GG_j=\GG_\infty^{(2)}$ with
 \begin{equation*}
      \GG_\infty^{(2)}(z):=\left\{  \begin{array}{ll}
	  \displaystyle\int_{B_1}|\nabla z|^2\,d\mu_a - (1+s)\int_{\partial B_1} z_\infty^2\,|x_n|^a\,d\mathcal{H}^{n-1} &	\qquad \mathrm{if}\, z\in \mathcal{B}^{(2)}_\infty\\
	  +\infty &  												\qquad \mathrm{otherwise},
			\end{array} \right.
 \end{equation*}
where $\mathcal{B}^{(2)}_\infty:=\{z\in z_\infty+ H^1_0(B_1,\mu_a)\,\,:\,\, z|_{B'^{,-}_1}=0\}$. 
We note that the third addendum of $\GG_j$ is zero in $\mathcal{B}^{(2)}_\infty$, 
while if $z\in\mathcal{B}_j\setminus\mathcal{B}^{(2)}_\infty$ the sequence $\GG_j(z)$ diverges; 
this heuristically justifies the choice of $\GG_\infty^{(2)}(z)$ and $\mathcal{B}^{(2)}_\infty$.\\
(3) If $\theta=+\infty$ and $\lim_j \GG_j(z_j)=+\infty$, then $\Gamma(L^2(B_1,\mu_a))$-$\lim\GG_j=\GG_\infty^{(3)}$ with
 \begin{equation*}
        \GG_\infty^{(3)}(z)=+\infty \qquad\qquad \mathrm{on}\,\,L^2(B_1,\mu_a).
 \end{equation*}

For the reader's convenience we recall the Definition of $\Gamma$-limit (see \cite{DalMaso}); 
the equality $\Gamma(L^2(B_1,\mu_a))$-$\lim\GG_j=\GG_\infty^{(i)}$ with $i=1,2,3$ is satisfied if the two following conditions hold:
\begin{itemize}
 \item[(a)] for all sequences $(w_j)_j\subset L^2(B_1,\mu_a)$ and $w\in L^2(B_1,\mu_a)$ such that $w_j\to  w$ in $L^2(B_1,\mu_a)$ it holds
 \begin{equation}\label{Gammaliminf}
  \liminf_j \GG_j(w_j)\geq \GG_j^{(i)}(w)
 \end{equation}
\item[(b)] for all $w\in L^2(B_1,\mu_a)$ there exists a sequence $(w_j)_j\subset L^2(B_1,\mu_a)$ such that $w_j\to w$ in $L^2(B_1,\mu_a)$ and
 \begin{equation}\label{Gammalimsup}
  \limsup_j \GG_j(w_j)\leq \GG_j^{(i)}(w).
 \end{equation}
\end{itemize}

\noindent \emph{Proof of the $\Gamma$-convergence: case (1)}.\\ 
(a) Without loss of generality we may suppose that $\liminf_j \GG_j(w_j)=\lim_j \GG_j(w_j)<+\infty$, then $w_j\in \mathcal{B}_j$
for all $j\in \N$. 
 Taking \eqref{cond_to_equicoerc_GGj} into account, we deduce
 \begin{equation*}
  \int_{B_1}|\nabla w_j|^2\,d\mu_a\leq \GG_j(w_j) + (1+s)\int_{\partial B_1} w_j^2\,|x_n|^a\,d\mathcal{H}^{n-1}<+\infty,  
 \end{equation*}
then, since $w_j\to w$ in $L^2(B_1,\mu_a)$ we have $\sup_j\|w_j\|_{H^1(B_1,\mu_a)}<+\infty$, 
so from Theorem~\ref{T:1.31 HKM} 
$\nabla w_j\rightharpoonup \nabla w$ in $L^2(B_1,\mu_a)$.
Then the respective traces converge in $L^2(\partial B_1,\mu_a)$ Theorem~\ref{T:traccia partialB1} 
and $L^2(B'_1)$ Theorem~\ref{T:tracciaFoc09}.
Hence, we obtain $(w+\theta h)|_{B'_1}\geq 0$ and, in particular,
since $w_j|_{B'_1}=z_j|_{B'_1}$ then $w|_{B'_1}=z_\infty|_{B'_1}$ and so $z_\infty\in \mathcal{B}^{(1)}_\infty$. At this point 
thanks to the convergence of traces of $w_j$ and weak semicontinuity of the norm of the gradient in $L^2(B_1,\mu_a)$ we have \eqref{Gammaliminf}.\\

\noindent(b) We observe that it is sufficient to prove the inequality for $w\in \mathcal{B}_\infty^{(1)}$ with
 \begin{equation}\label{w-zinfty suppcomp}
  \mathrm{supp}(w-z_\infty)\subset B_\rho \qquad \mathrm{for\,\, some}\,\,\rho\in(0,1).
 \end{equation}
If we want to deal with the general case, we consider the function
 \begin{equation*}
  w_t(x)=t^{1+s}\left(w\left(\frac{x}{t}\right)\chi_{B_1}\left(\frac{x}{t}\right)+z_\infty\left(\frac{x}{t}\right)\chi_{B_{1/t}\setminus \overline{B_1}}\left(\frac{x}{t}\right)\right)\quad \mathrm{with}\,\,t<1.
 \end{equation*}
It is easy to prove that $w_t\in H^1(B_1,\mu_a)$ and $\mathrm{supp}(w_t-z_\infty)\subset B_t$; moreover, $w_t\to w$ in $H^1(B_1,\mu_a)$ 
(for a similar procedure see \cite[Proposition~2.4.1, Chapter~2]{Geraci-Thesis}).
If \eqref{Gammalimsup} holds for all $w_t$, resorting to a diagonalization argument we obtain \eqref{Gammalimsup} for $w$.
Therefore for a Uryshon's type property it is sufficient to prove the following property: fixing $w$ as in \eqref{w-zinfty suppcomp}, 
for all sub sequences $j_k\uparrow + \infty$ there exists an extract subsequence $j_{k_l}\uparrow + \infty$ 
and there exists $w_l\to w$ in $L^2(B_1,\mu_a)$ such that 
\footnote{Let us suppose by contradiction that there exists $w$ such that 
\begin{equation*}
 \Gamma\mathit{-}\limsup_j\GG_j(w)>\GG_\infty^{(1)}(w),
\end{equation*}
if $(w_j)_{j\in\N}$ is a sequence that achieves the $\Gamma$-$\limsup$, i.e. $\limsup_j\GG_j(w_j)=\Gamma$-$\limsup_j\GG_j(w_j)$,
and $j_k$ is a subsequence for which $\limsup_j\GG_j(w_j)=\limsup_k\GG_{j_k}(w_{j_k})$, 
by assumption then there exists $j_{k_l}$ such that
\begin{equation*}
 \lim_l\GG_{j_{k_l}}(w_{j_{k_l}})\leq \GG_\infty^{(1)}(w),
\end{equation*}
leading to a contradiction.
}
\begin{equation*}
\limsup_l\GG_{j_{k_l}}(w_l)\leq \GG_{\infty}^{(1)}(w).
\end{equation*}

Setting $r\in(\rho,1)$ let $R:=\frac{1+r}{2}$ and let $\phi\in C^1_c(B_1)$ be a cut-off function such that 
\begin{equation*}
 \phi|_{B_r}\equiv 1, \qquad  \phi|_{B_1\setminus \overline{B_R}}\equiv 0, \qquad \|\nabla \phi\|_{L^\infty}\leq \frac{4}{1-r}.
\end{equation*}
We define
\begin{equation}\label{e:wkr}
 w_k^r :=\phi\left(w+(\theta-\theta_{j_k})h\right) + (1-\phi)z_{j_k}
\end{equation}
and we verify that $w_k^r\in\mathcal{B}_{j_k}$. In fact $w\in\mathcal{B}_\infty^{(1)}$, $z_{j_k}\in\mathcal{B}_{j_k}$ and
\begin{equation*}
 w_k^r + \theta_{j_k}h= \phi(w + \theta h) + (1-\phi)(z_{j_k} + \theta_{j_k}h)\geq 0. 
\end{equation*}
Therefore, since $\theta_{j_k}\to \theta\in[0,+\infty)$ we have $w_k^r\to \phi w + (1-\phi)z_\infty$ in $L^2(B_1,\mu_a)$.
Thanks to the convergence of traces of $z_{j_k}$ in $L^2(B_1')$ it is enough to prove the upper bound inequality
for the first addendum of
$\GG_j$ and $\GG_\infty^{(1)}$ respectively. From \eqref{e:wkr}, we can infer 
\begin{equation}\label{{||nabla wrk||<}}
\begin{split}
 \int_{B_1}|\nabla w_k^r|^2\,d\mu_a \leq &\int_{B_r}|\nabla w +(\theta-\theta_{j_k})\nabla h|^2\,d\mu_a
 +\underbrace{\int_{B_R\setminus \overline{B_r}}|\nabla w_k^r|^2\,d\mu_a}_{:=I_k} + \int_{B_1\setminus \overline{B_R}}|\nabla z_{j_k}|^2\,d\mu_a. 
\end{split}
\end{equation}
Since $r>\rho$, from assumption \eqref{w-zinfty suppcomp}, we estimate the term $I_k$ as follows
\begin{equation*}
\begin{split}
 I_k\leq &3\int_{B_R\setminus \overline{B_r}} \phi^2|\nabla w+(\theta-\theta_{j_k})\nabla h|^2\,d\mu_a\\ 
 &+ 3\int_{B_R\setminus \overline{B_r}} (1-\phi)^2|\nabla z_{j_k}|^2\,d\mu_a + 3\int_{B_R\setminus \overline{B_r}} |\nabla\phi|^2|z_\infty-z_{j_k}+(\theta-\theta_{j_k})\nabla h|^2\,d\mu_a 
\end{split}
\end{equation*}
So 
\begin{equation}\label{limsup||nabla wrk||<}
\begin{split}
 \limsup_k \int_{B_1}|\nabla w_k^r|^2\,d\mu_a
 \leq \int_{B_r}|\nabla w|^2\,d\mu_a +3\int_{B_R\setminus \overline{B_r}}|\nabla w|^2\,d\mu_a + 4\limsup_k\,\int_{B_1\setminus \overline{B_r}}|\nabla z_{j_k}|^2\,d\mu_a
\end{split}
\end{equation}
By the $(1+s)$-homogeneity of $z_{j_k}$, we deduce 
\begin{equation*}
\begin{split}
 \int_{B_1\setminus \overline{B_r}}&|\nabla z_{j_k}|^2\,d\mu_a
 =\int^{1}_{r}\int_{\partial B_t} |\nabla z_{j_k}|^2\,|x_n|^a\,d\mathcal{H}^{n-1}\,dt\\
 =&\int^{1}_{r}t^n\int_{\partial B_1} |\nabla z_{j_k}|^2\,|x_n|^a\,d\mathcal{H}^{n-1}\,dt
 =\frac{1-r^{n+1}}{n+1}\int_{\partial B_1} |\nabla z_{j_k}|^2\,|x_n|^a\,d\mathcal{H}^{n-1}
\end{split}
\end{equation*}
which leads us to
\begin{equation*}
\begin{split}
 \int_{\partial B_1} |\nabla z_{j_k}|^2\,|x_n|^a\,d\mathcal{H}^{n-1}=\frac{n+1}{1-(1/2)^{n+1}}\int_{B_1\setminus \overline{B_r}}|\nabla z_{j_k}|^2\,d\mu_a \stackrel{\eqref{|z_j|=1}}{\leq} 2(n+1)
\end{split}
\end{equation*}
in turn implying
\begin{equation}\label{e:int nabla zjk su B1-Br}
\begin{split}
 \int_{B_1\setminus \overline{B_r}}&|\nabla z_{j_k}|^2\,d\mu_a
 \leq 2\,(1-r)\,(n+1).
\end{split}
\end{equation}
We apply this construction to a subsequence $r_l\uparrow 1$ and $R_l:=\frac{1+r_l}{2}$ and with a diagonal argument we obtain a
subsequence $w_l\to w$ in $L^2(B_1,\mu_a)$. 
Thanks to \eqref{limsup||nabla wrk||<} and \eqref{e:int nabla zjk su B1-Br}
\begin{equation*}
 \begin{split}
  \limsup_l &\int_{B_1}|\nabla w_l|^2\,d\mu_a
  \leq \int_{B_1}|\nabla w|^2\,d\mu_a + 3\limsup_l\int_{B_{R_l}\setminus \overline{B_{r_l}}}|\nabla w|^2\,d\mu_a + 4 \limsup_l \int_{B_1\setminus \overline{B_{r_l}}}|\nabla z_{j_l}|^2\,d\mu_a\\
  &\leq \int_{B_1}|\nabla w|^2\,d\mu_a + \lim_l 8\,(1-r_l)(n+1)
  = \int_{B_1}|\nabla w|^2\,d\mu_a,
 \end{split}
\end{equation*}
and this provides the conclusion.

\noindent\textit{Proof of the $\Gamma$-convergence: case (2).}\\
(a) Without loss of generality we assume that
 \begin{equation}\label{liminf=lim(2)}
  \liminf_j\GG_j(w_j)=\lim_j\GG_j(w_j) <+\infty.
 \end{equation}
 Let $w_j\to w$ in $L^2(B_1,\mu_a)$, since $w_j\in \mathcal{B}_j$ and \eqref{liminf=lim(2)}, then $w\geq 0$ on $B'^{,-}_1$.
 From \eqref{-wdh>0}, we obtain
 \begin{align*}
 0&\leq -\theta_j\int_{B'_1}w_j\RR(h)\,d\mathcal{H}^{n-1}
 \leq \GG_j(w_j)+(1+s)\int_{\partial B_1}z_j^2\,|x_n|^a\,d\mathcal{H}^{n-1}\\
 &\leq \sup_j\left(\GG_j(w_j)+(1+s)\int_{\partial B_1}z_j^2\,|x_n|^a\,d\mathcal{H}^{n-1}\right)<+\infty.
 \end{align*}
Then dividing by $\theta_j$, the convergence of traces leads us to
\begin{equation*}
 \int_{B'_1}w\RR(h)\,d\mathcal{H}^{n-1}=\lim_j \int_{B'_1}w_j\RR(h)\,d\mathcal{H}^{n-1}=0
\end{equation*}
From \eqref{lim_eh} we deduce that $w|_{B'^{,-}_1}=0$, or rather $w\in \mathcal{B}_\infty^{(2)}$. 
In particular also $z_\infty\in \mathcal{B}_\infty^{(2)}$ because $\sup_j\GG_j(z_j)<+\infty$. 
Then, according to the semicontinuity of the norm $H^1(B_1,\mu_a)$ with respect to weak convergence of gradient,
the convergence of $w_j$ in $L^2(B_1,\mu_a)$ and the convergence of traces in $L^2(\partial B_1,|x_n|^a\mathcal{H}^{n-1})$ 
we obtain the $\Gamma$-$\liminf$ inequality \eqref{Gammaliminf}.\\

\noindent(b) Now we prove the inequality \eqref{Gammalimsup}. With the same argument used in case (1) we can consider the case of 
$w\in \mathcal{B}_\infty^{(2)}$ for which \eqref{w-zinfty suppcomp} holds and for which for all $j_k\uparrow +\infty$ we find a subsequence
$j_{k_l}\uparrow +\infty$ and a sequence $w_l\to w$ in $L^2(B_1,\mu_a)$ such that
 \begin{equation}\label{limsup=lim(2)}
   \limsup_l\GG_{j_{k_l}}(w_l)\leq \GG_\infty^{(2)}.
  \end{equation}
We introduce the positive Radon measures 
\begin{equation*}
 \nu_k:= |\nabla z_{j_k}|^2\,|x_n|^a\,\mathcal{L}^n\llcorner B_1 
	- 4\theta_{j_k}(z_{j_k}+\theta_{j_k}h)\,\RR(h)\mathcal{H}^{n-1}\llcorner B'^{,-}_1.
\end{equation*}
Assuming that  $k>>1$, we obtain
\begin{align*}
 \nu_k(B_1)=\GG_{j_k}(z_{j_k}) + (1+s)\int_{\partial B_1} z_{j_k}^2\,|x_n|^a\, d\mathcal{H}^{n-1}
	  \leq \sup_j \GG_j(z_j) + C \sup_j \|z_j\|_{H^1(B_1,\mu_a)}<\infty,
\end{align*}
which leads us to
\begin{equation*}
 \sup_k \nu_k(B_1)= \Lambda_0<+\infty.
\end{equation*}
In order to prove $\nu_k(B_\rho)=\rho^{n+1} \nu(B_1)$ we observe that setting $\rho\in (0,1)$ by $(1+s)$-homogeneity 
of $z_{j_k}$ we obtain 
\begin{equation*}
 \begin{split}
  \int_{B_\rho}|\nabla &z_{j_k}|^2\,d\mu_a = \int_0^\rho dt \int_{\partial B_t} |\nabla z_{j_k}|^2\,|x_n|^a\,d\mathcal{H}^{n-1}
  \stackrel{x=ty}{=}\int_0^\rho t^{n-1} \int_{\partial B_1} |\nabla z_{j_k}(ty)|^2\,|ty_n|^a\,d\mathcal{H}^{n-1}(y)\,dt\\
  = &\int_0^\rho t^{n} \int_{\partial B_1} |\nabla z_{j_k}(y)|^2\,|y_n|^a\,d\mathcal{H}^{n-1}(y)\,dt
  = \frac{\rho^{n+1}}{n+1} \int_{\partial B_1} |\nabla z_{j_k}(y)|^2\,|y_n|^a\,d\mathcal{H}^{n-1}(y)\,dt\\
  = &\rho^{n+1} \int_0^1 t^{n}\,dt \int_{\partial B_1} |\nabla z_{j_k}(y)|^2\,|y_n|^a\,d\mathcal{H}^{n-1}(y)\,dt\\
  \stackrel{ty=x}{=} &\rho^{n+1} \int_0^1 \int_{\partial B_t} |\nabla z_{j_k}(x)|^2\,|x_n|^a\,d\mathcal{H}^{n-1}(y)\,dt 
  = \rho^{n+1} \int_{B_1} |\nabla z_{j_k}|^2\,d\mu_a,
 \end{split}
\end{equation*}
and
\begin{equation*}
 \begin{split}
  \int_{B'_\rho}z_{j_k}\,\RR(h)(\wx)d\mathcal{H}^{n-1}
  = &\int_0^\rho \int_{\partial B'_t}z_{j_k}\,\RR(h)(\wx)d\mathcal{H}^{n-2}\\
  \stackrel{\wx=t\widehat{y}}{=}
  &\int_0^\rho t^{n-2}\,dt\int_{\partial B'_1}z_{j_k}(t\widehat{y},0)\,\lim_{\varepsilon\to 0}(t\varepsilon)^a\frac{\partial h}{\partial x_n}(t\widehat{y},t\varepsilon)d\mathcal{H}^{n-2}(\widehat{y})\\
  = &\frac{\rho^{n+1}}{n+1} \int_{\partial B'_1} z_{j_k}(\widehat{y},0)\,\RR(h)(\widehat{y})d\mathcal{H}^{n-2}(\widehat{y})
  = \rho^{n+1}\int_{B'_1}z_{j_k}\,\RR(h)(\wx)d\mathcal{H}^{n-1}
 \end{split}
\end{equation*}
where in the last equality we did the previous calculus again in reverse order.
Since $\nu_k(B_1)<\infty$ then $\nu_k(\partial B_\rho)=0$ with $\rho\in(0,1)\setminus I$ where $I$ is a set at the most countable. 
Thus
\begin{equation}\label{e:nuk(B1-B2)}
 \nu_k(B_{\rho_1}\setminus B_{\rho_2})\leq \Lambda_0 (\rho_1^{n+1}-\rho_2^{n+1})\leq c(n,\Lambda_0)(\rho_1-\rho_2),
\end{equation}
for all $0<\rho_1\leq\rho_2<1$ such that $\rho_1,\rho_2\in (0,1)\setminus I$.
Repeating the argument in \eqref{w-zinfty suppcomp} we prove the $\Gamma$-$\limsup$ inequality for function 
$w\in \mathcal{B}_\infty^{(2)}$ for which there exists some $\rho\in (0,1)$ 
such that $\{w\not\equiv z_\infty\}\subset\subset B_\rho$.
We extend $w$ on $\R^n$ as $z_\infty$ in $B_\rho^c$ and we indicate the extension by $w$ again.
We fix $\varepsilon>0$ and introduce the following auxiliary tools.\\ 
Due to the definition of $H^1(B_1,\mu_a)$ as $\overline{C^\infty(B_1)}^{\|\cdot\|_{H^1(B_1, \mu_a)}}$ 
(cf. \cite[Section~1.9 and Lemma~1.15]{HKM})
there exists a function $v_\d\in C^\infty(B_1)$ such that
\begin{equation}\label{e:vd densità}
 \|v^\d - w\|_{H^1(B_1,\mu_a)}<\d(\varepsilon) \qquad\qquad \textit{with $\d(\varepsilon)=o(\varepsilon)$.}
\end{equation}
Let $w^\varepsilon(x):=w(x-3\varepsilon e_{n-1})$ be the translated function along the direction $e_{n-1}$. 
Since $w\in \mathcal{B}_\infty^{(2)}$, we observe that
\begin{equation*}
 w^\varepsilon(x)=0 \quad\Longleftrightarrow\quad x-3\varepsilon e_{n-1} \in \{(\wx,0)\,\,:\,\,x_{n-1}\leq 0\} 
 \quad\Longleftrightarrow\quad x\in \{(\wx,0)\,\,:\,\,x_{n-1}\leq 3\varepsilon\}.
\end{equation*}
Let $I_\sigma$ be the set defined as
\begin{equation}\label{d:Ie}
  \begin{split}
   I_\sigma = \{x\in B_1\,\,:\,\, \mathrm{dist}(x,B_1^{',-})<\sigma\}
  \end{split}
\end{equation}
Let $\varphi_\varepsilon$ and $\chi_\varepsilon$ be two cut-off functions such that 
\begin{equation}\label{d:varphie e chie}
  \begin{split}
   \varphi_\varepsilon\in C^\infty_c(I_{3\varepsilon}), \quad \varphi_{\varepsilon\,\,|I_{2\varepsilon}}\equiv 1, \quad \|\nabla \varphi_\varepsilon\|_{L^\infty(B_1)}\leq \frac{C}{\varepsilon}\\
   \chi_\varepsilon\in C^\infty_c(B_{1-\varepsilon}), \quad \chi_{\varepsilon\,\,|B_{1-2\varepsilon}}\equiv 1, \quad \|\nabla \chi_\varepsilon \|_{L^\infty(B_1)}\leq \frac{C}{\varepsilon}.
  \end{split}
\end{equation}
For all $0<\varepsilon<<1$ we build the sequence of functions
\begin{equation*}
 w_k^{(\varepsilon)}:=\chi_\varepsilon (\varphi_\varepsilon w^\varepsilon + (1-\varphi_\varepsilon)v^\d) + (1-\chi_\varepsilon)z_{j_k}.
\end{equation*}
Then we can at once infer
\begin{equation*}
 w_k^{(\varepsilon)}\in z_{j_k} + W^{1,2}_0(B_1)
\end{equation*}
and since we can write 
\begin{equation*}
 w_k^{(\varepsilon)}+\theta_{j_k}h:=\chi_\varepsilon (\varphi_\varepsilon (v_\d+\theta_{j_k}h) + 
					(1-\varphi_\varepsilon)(w^\tau+\theta_{j_k}h)) + (1-\chi_\varepsilon)(z_{j_k}+\theta_{j_k}h),
\end{equation*}
we prove that $w_k^{(\varepsilon)}\in \mathcal{B}_{j_k}$: $w_k^{(\varepsilon)}$ is a convex combination of functions
$v_\d$, $w^\tau$ and $z_{j_k}$ with boundary data as $z_{j_k}$ and every addendum is bigger than $-\theta_{j_k}h$ restricted to $B'_1$.
In fact 
\begin{itemize}
 \item[(i)] by definition $z_{j_k}+\theta_{j_k}h\geq 0$ in $B'_1$;
 \item[(ii)] if $x\in \mathrm{supp}(\varphi_\varepsilon)\cap B'_1$ then $x_{n-1}<3\varepsilon$. 
	Thus $w^\varepsilon(x)=0$ then $\varphi_\varepsilon(x)(w^\tau(x)+\theta_{j_k}h(x))=\varphi_\varepsilon(x)\theta_{j_k}h(x)\geq 0$;
 \item[(iii)] if $x\in \mathrm{supp}(1-\varphi_\varepsilon)\cap B'_1$ then $x_{n-1}\geq 2\varepsilon$, so $h(\wx,0)>0$ and as
       $\theta_{j_k}\to +\infty$ $v^\d(x)+\theta_{j_k}h(x)\geq -\|v_\d\|_{L^\infty(B_1)}+\theta_{j_k}h(x)\geq0$ for $k>k_\d$.
\end{itemize}
So $w_k^{(\varepsilon)}\in \mathcal{B}_{j_k}$ for $k>k_\d$.\\
Next, consider, 
\begin{align*}
 J_k^\varepsilon &:= -4\theta_{j_k}\int_{B'_1} w_k^{(\varepsilon)} \RR(h)\,d\mathcal{H}^{n-1} \\
 I_k^\varepsilon &:=  \int_{B_1} |\nabla w_k^{(\varepsilon)}|^2\, d\mu_a,
\end{align*}
respectively the trace term and the volume term of the energy of $w_k^{(\varepsilon)}$.
By definition we have
\begin{equation*}
 \begin{split}
  J_k^\varepsilon \leq -4\theta_{j_k}\!\!\int_{B'_{1-\varepsilon}}\!\!\!\!\!\!\! (\varphi_\varepsilon w^\varepsilon + (1-\varphi_\varepsilon)v^\d) \RR(h)\,d\mathcal{H}^{n-1}
                       -4\theta_{j_k}\!\!\int_{B'_1\setminus B'_{1-2\varepsilon}}\!\!\!\!\!\!\!\!\!\! z_{j_k} \RR(h)d\mathcal{H}^{n-1} 
                       = J_k^{(1)}\!+\!J_k^{(2)}.
 \end{split} 
\end{equation*}
According to (i), \eqref{complementarity} and \eqref{e:nuk(B1-B2)} we deduce
\begin{equation}\label{e:stima Jk2}
 0\leq \sup_k J_k^{(2)}\leq \sup_k \nu_k(B_1\setminus B_{1-2\varepsilon}) \leq C\,2\varepsilon.
\end{equation}
Instead, due to (ii), the function $w^\varepsilon_{|B'_1\cap G_{3\varepsilon}}=0$ and
from definitions of $I_{2\varepsilon}$ and $h$ we have $\RR(h)_{|B'_{1-\varepsilon}\setminus I_{2\varepsilon}}=0$.
From this we infer 
\begin{equation}\label{e:stima Jk1}
 0\leq J_k^{(1)}\leq -4\theta_{j_k}\left(\int_{B'_{1-\varepsilon}\cap I_{3\varepsilon}} w^\varepsilon \RR(h)d\mathcal{H}^{n-1} 
	+ \int_{B'_{1-\varepsilon}\setminus I_{3\varepsilon}}v^\d \RR(h)d\mathcal{H}^{n-1}\right)=0.
 \end{equation}
Putting  \eqref{e:stima Jk2} and \eqref{e:stima Jk1} together yields
\begin{equation}
 \limsup_{k\to\infty} J_k^\varepsilon \leq C\varepsilon.
\end{equation}
In order to estimate the functional $I_k^\varepsilon$ we observe that
\begin{equation*}
\begin{split}
I_k^\varepsilon&\leq \int_{B_{1-2\varepsilon}}\!\!\!\!\! |\nabla (\varphi_\varepsilon w^\varepsilon + (1-\varphi_\varepsilon)\,v^\d)|^2\, d\mu_a
		   + c\int_{B_{1-\varepsilon}\setminus B_{1-2\varepsilon}}\!\!\!\!\!\!\!\!\!\!|\nabla(\varphi_\varepsilon w^\varepsilon +(1-\varphi_\varepsilon) v^\delta)|^2\, d\mu_a\\
		   &+ c \int_{B_{1}\setminus B_{1-2\varepsilon}}\!\!\!\!\!\!|\nabla z_{j_k}|^2\, d\mu_a
		   +\frac{c}{\varepsilon^2} \int_{B_{1-\varepsilon}\setminus B_{1-2\varepsilon}}\!\!\!\!\!\!|(\varphi_\varepsilon w^\varepsilon +(1-\varphi_\varepsilon) v^\d - z_{j_k}|^2\, d\mu_a
		   = I_k^{(1)} + I_k^{(2)} + I_k^{(3)} + I_k^{(4)}.
\end{split}
\end{equation*}
We estimate the four addenda separately.
From condition \eqref{e:nuk(B1-B2)}, we can infer
\begin{equation}\label{e:stima I3}
 \sup_k I_k^{(3)}\leq\sup_k \nu_k(B_{1}\setminus B_{1-2\varepsilon}) C\,\varepsilon.
\end{equation}
We now estimate the first term; recalling that ${\varphi_\varepsilon}_{|I_{3\varepsilon}^c}=0$
\begin{equation}\label{e:stima I1 prel}
 \begin{split}
  I_k^{(1)}=&\int_{B_{1-2\varepsilon}\setminus I_{3\varepsilon}}|\nabla v^\d|^2\,d\mu_a + \int_{B_{1-2\varepsilon}\cap I_{3\varepsilon}}\left|\nabla\left(\varphi_\varepsilon(w^\varepsilon-v^\d)\right)\nabla v^\d\right|^2\,d\mu_a\\
  \leq &\int_{B_{1-2\varepsilon}\setminus I_{3\varepsilon}}|\nabla v^\d|^2\, d\mu_a + c\int_{B_{1-2\varepsilon}\cap I_{3\varepsilon}}|\nabla v^\d|^2\, d\mu_a\\
      &+ \int_{B_{1-2\varepsilon}\cap I_{3\varepsilon}}|\nabla (w^\varepsilon-v^\d)|^2\, d\mu_a + \frac{c}{\varepsilon^2}\int_{B_{1-2\varepsilon}\cap I_{3\varepsilon}}|v^\d-w^\varepsilon|^2\, d\mu_a\\
      \leq& \int_{B_{1-2\varepsilon}\setminus I_{3\varepsilon}}|\nabla v^\d|^2\, d\mu_a 
	+ c\int_{B_{1-2\varepsilon}\cap I_{3\varepsilon}}|\nabla (v^\d - \nabla w)|^2\, d\mu_a\\
	  &+ c \int_{B_{1-2\varepsilon}\cap I_{3\varepsilon}} |\nabla w|^2\, d\mu_a
	      +\frac{c}{\varepsilon^2} \int_{B_{1-2\varepsilon}\cap I_{3\varepsilon}}(|v^\d-w|^2 + |w-w^\varepsilon|^2)\, d\mu_a.
 \end{split}
\end{equation}
Taking the last addendum above into account, we notice that for all $\phi$ smooth functions and~$\tau>0$
\begin{equation*}
 |\phi(x-\tau e_{n-1}) - \phi(x)|\leq \tau\int_0^1|\nabla \phi|(x-\tau t e_{n-1})\,dt.
\end{equation*}
Then, by a simple application of Fubini’s theorem we deduce
\begin{equation*}
 \frac{c}{\varepsilon^2} \int_{B_{1-2\varepsilon}\cap I_{3\varepsilon}}|\phi(x-\tau e_{n-1}) - \phi(x)|^2\, d\mu_a
 \leq c \frac{\tau^2}{\varepsilon^2} \int_{(B_{1-2\varepsilon}\cap I_{3\varepsilon})+[0,\tau]e_{n-1}}|\nabla \phi|^2 \, d\mu_a
\end{equation*}
where $(B_{1-2\varepsilon}\cap I_{3\varepsilon})+[0,\tau]e_{n-1}$ denotes the Minkowski sum between sets.
So, thanks to a density argument and for $\tau=3\varepsilon$ we infer
\begin{equation*}
 \frac{c}{\varepsilon^2} \int_{G_{2\varepsilon} \setminus G_{3\varepsilon}}|w-w^\varepsilon|^2\, d\mu_a
 \leq c  \int_{(B_{1-2\varepsilon}\cap I_{3\varepsilon})+[0,\tau]e_{n-1}}|\nabla w|^2 \, d\mu_a.
\end{equation*}
So, from \eqref{e:stima I1 prel}, according to 
\eqref{e:vd densità}, the continuity of translation in $L^2$ and the absolute continuity 
of the integral, and observing that $\mathcal{L}^n((B_{1-2\varepsilon}\cap I_{3\varepsilon})+[0,\tau]e_{n-1})=O(\varepsilon)$ 
we obtain
\begin{equation}\label{e:stima I1}
 \begin{split}
  I_k^{(1)}
	    \leq \int_{B_{1-2\varepsilon}\setminus I_{3\varepsilon}}|\nabla v_{\varepsilon}|^2\, d\mu_a + O(\varepsilon).   
 \end{split}
\end{equation}
Reasoning in the same way as in the estimate of $I_k^{(1)}$ we obtain
\begin{equation}\label{e:stima I2}
 I_k^{(2)}\leq O(\varepsilon)
\end{equation}
Since $\mathrm{supp}\varphi_\varepsilon\subset G_{2\varepsilon}$ and recalling that by condition \eqref{w-zinfty suppcomp}, 
if we choose $\varepsilon$ sufficiently small such that $\rho<1-5\varepsilon$, 
$\mathrm{supp}(w^{3\varepsilon} -z_\infty^{3\varepsilon})\subset B_{1-2\varepsilon}$, we obtain
\begin{equation*}
\begin{split}
 I_k^{(4)}&\leq \frac{c}{\varepsilon^2}\int_{B_{1-\varepsilon}\setminus B_{1-2\varepsilon}} |\varphi_\varepsilon(w^\varepsilon-v^\d)|^2\, d\mu_a
	  +\frac{c}{\varepsilon^2}\int_{B_{1-\varepsilon}\setminus B_{1-2\varepsilon}} |v^\d-z_{j_k}|^2\, d\mu_a\\
	  &\leq \frac{c}{\varepsilon^2}\int_{B_{1-\varepsilon}\setminus B_{1-2\varepsilon}} (|w^\varepsilon-w|^2+|w-v^\d|^2+|w-z_{j_k}|^2)\, d\mu_a
\end{split}
\end{equation*}
So, proceeding as in estimate of $I_k^{(1)}$ and recalling that $\mathrm{supp}(w-z_\infty)\subset B_\rho$ for $\varepsilon$ sufficiently small we deduce 
\begin{equation}\label{e:stima I4}
\begin{split}
\limsup_{k\to\infty} I_k^{(4)}
\leq \limsup_{k\to\infty} \frac{c}{\varepsilon^2}\int_{B_{1-\varepsilon}\setminus B_{1-2\varepsilon}} |z_\infty-z_{j_k}|^2\,d\mu_a + O(\varepsilon)
\leq O(\varepsilon).
\end{split}
\end{equation}
Then putting together estimates in \eqref{e:stima I3}, \eqref{e:stima I1}, \eqref{e:stima I2} and \eqref{e:stima I4} leads to
\begin{equation*}
 \begin{split}
 \limsup_{k\to\infty} I_k^\varepsilon \leq \int_{B_{1-2\varepsilon}\setminus I_{3\varepsilon}}|\nabla v_{\varepsilon}|^2\, d\mu_a 
												+O(\varepsilon).
\end{split}
\end{equation*}
So, since 
\begin{equation*}
 \begin{split}
  w^{(\varepsilon)}_k \xrightarrow{k\to\infty} \chi_\varepsilon (\varphi_\varepsilon v_{\d_\varepsilon}+ (\varphi_\varepsilon)w^{3\varepsilon}) + (1-\chi_\varepsilon) z_\infty =: w^{(\varepsilon)}\qquad \textit{in }\, L^2(B_1,\mu_a)
 \end{split}
\end{equation*}
and 
\begin{equation*}
 \begin{split}
  w^{(\varepsilon)} \xrightarrow{\varepsilon\to 0} w\qquad \textit{in }\, L^2(B_1,\mu_a),
 \end{split}
\end{equation*}
we conclude by the lower semicontinuity of the $\Gamma$-$\limsup$
\begin{equation*}
 \begin{split}
  \Gamma\mathit{-}\limsup_{k\to \infty} \GG_{j_k}(w)&\leq \liminf_{\varepsilon\to 0}\left(\Gamma\mathit{-}\limsup_{k\to \infty} \GG_{j_k}(w^{(\varepsilon)})\right)\\
  &\leq \limsup_{\varepsilon\to 0}\left(\limsup_{k\to \infty} (I_k^\varepsilon+J_k^{\varepsilon})\right)
  \leq \int_{B_1}|\nabla w|^2 \, d\mu_a,
 \end{split}
\end{equation*}
that provides the thesis.\\

\noindent\textit{Proof of the $\Gamma$-convergence: case (3).}\\ 
(a) From \eqref{quasiminimality z_j GG}, we immediately have
 \begin{equation*}
  \liminf_j \GG_j(w_j)\geq \liminf_j (1-\kappa_j)\GG_j(z_j) =+\infty= \GG_\infty^{(3)}.
 \end{equation*}
 (b) This is trivial, in fact $\liminf_j\GG_j(w_j)\leq+\infty= \GG_\infty^{(3)}$.

\textbf{Step 4:\,Improving the convergence of $(z_j)_j\in\N)$ if $\lim_j\GG_j(z_j)<+\infty$.}
Using a standard result of $\Gamma$-convergence we show that $z_j\to z_\infty$ in $H^1(B_1,\mu_a)$.\\
For equi-coercivity of $\GG_j$ seen in \eqref{cond_to_equicoerc_GGj}, \cite[Lemma 2.10]{CSS} (a version of Poincaré inequality for weighted Sobolev spaces)
and $\|z_j\|_{H^1(B_1,\mu_a)}=1$ we have
\begin{equation*}
 \|w\|_{H^1(B_1,\mu_a)}\leq C \sqrt{\GG_j(w)+1},
\end{equation*}
so every minimizing sequence converges weakly in $H^1(B_1,\mu_a)$ and thanks to Theorem~\ref{T:8.1 HK}
converges strongly in $L^2(B_1,\mu_a)$. Since $\GG_j$ is semicontinuous with respect to weak topology of $H^1(B_1,\mu_a)$ there exists $\zeta_j$
minimizer of $\GG_j$. Taking into account \cite[Theorem 7.8]{DalMaso},
with $i=1,2$ there exists $\zeta_\infty\in H^1(B_1,\mu_a)$ such that
\begin{align}
 &\zeta_j\to \zeta_\infty,\qquad\qquad \textit{in }\,L^2(B_1,\mu_a)\\ \label{GG_j(zeta_j)to}
 &\GG_j(\zeta_j)\to \GG^{(i)}_\infty(\zeta_\infty),\\ \label{uniqueness zeta_infty}
 &\zeta_\infty \,\,\text{is the unique minimizer of}\,\, \GG^{(i)}_\infty, 
\end{align}
where due to \eqref{uniqueness zeta_infty} we have used the strict convexity of $\GG^{(i)}_\infty$.
Therefore using the strong convergence of traces in $L^2(\partial B_1,|x_n|^a\mathcal{H}^{n-1})$ and $L^2(B_1')$, then from
the estimates
\begin{equation}\label{GG_j(zeta_j)<}
 \GG_j(\zeta_j)\leq \GG_j(z_j)\leq \sup_j \GG_j(z_j)<\infty,
\end{equation}
and \eqref{uniqueness zeta_infty} we obtain
\begin{equation*}
 \int_{B_1}|\nabla \zeta_j|^2\,d\mu_a\to \int_{B_1}|\nabla \zeta_\infty|^2\,d\mu_a,
\end{equation*}
which implies $\zeta_j \to \zeta_\infty$ in $H^1(B_1,\mu_a)$.
According to \eqref{quasiminimality z_j GG} and \eqref{GG_j(zeta_j)<}, $z_j$ is an almost minimizer of $\GG_j$ in the following sense 
\begin{equation*}
 0\leq \GG_j(z_j)-\GG_j(\zeta_j)\leq \kappa_j \GG_j(z_j)\leq \kappa_j \sup_j \GG_j(z_j). 
\end{equation*}
Since $\kappa_j\downarrow 0$ and $z_j \rightharpoonup z_\infty$ in $H^1(B_1,\mu_a)$, \eqref{GG_j(zeta_j)to} and Step $3$ yield that
\begin{equation}\label{e:GGinfty z<GGinfty zeta}
 \GG_\infty^{(i)}(z_\infty)\leq \liminf_j \GG_j(z_j)= \lim_j \GG_j(\zeta_j)= \GG_\infty^{(i)}(\zeta_\infty),
\end{equation}
with $i=1,2$. From \eqref{quasiminimality z_j GG}, we infer 
\begin{equation*}
 \GG_j(z_j)\leq \frac{1}{1-k_j} \GG_j(\zeta_j);
\end{equation*}
from this, by \eqref{e:GGinfty z<GGinfty zeta} and by strong convergence of traces we obtain
\begin{equation*}
 \liminf_j \int_{B_1}|\nabla z_j|^2\,d\mu_a= \int_{B_1}|\nabla z_\infty|^2\,d\mu_a,
\end{equation*}
that with the weak convergence of in $H^1(B_1,\mu_a)$ proves the convergence
\begin{equation*}
 z_j\to z_\infty\qquad\qquad \textit{in }\, H^1(B_1,\mu_a).
\end{equation*}
In particular 
\begin{equation}\label{||zinfty||=1}
 \|z_\infty\|_{H^1(B_1,\mu_a)}=1.
\end{equation}

\textbf{Step 5: Case (1) cannot occur.}
We recall properties of $z_\infty$:
\begin{itemize}
 \item[(i)] $\|z_\infty\|_{H^1(B_1,\mu_a)}=1$;
 \item[(ii)] $z_\infty$ is $(1+s)$-homogeneous and even with respect to $\{x_n=0\}$;
 \item[(iii)] $z_\infty$ is the unique minimizer of $\GG_\infty^{(1)}$ with respect to its boundary data;
 \item[(iv)] $z_\infty\in \mathcal{B}_\infty^{(1)}=\{z\in z_\infty+H_0^1(B_1,\mu_a) \,:\, (z + \theta h)|_{B_1'}\geq 0\}$.
\end{itemize}
These properties imply that
\begin{equation*}
 w_\infty:=z_\infty + \theta h
\end{equation*}
is the minimizer of $\int_{B_1} |\nabla\cdot|^2\,d\mu_a$ among all functions $w\in w_\infty + H^1_0(B_1,\mu_a)$ and $w|_{B_1'}\geq 0$
in the sense of the trace. So, $w_\infty$ is the solution of the fractional obstacle problem.
To prove this claim, for all $z\in \mathcal{B}_\infty^{(1)}$ we consider $w:=z +\theta h$ and, recalling \eqref{dGG}, we have
\begin{equation*}
 \begin{split}
  \GG_\infty^{(1)}(z)=&\int_{B_1}|\nabla w|^2\,d\mu_a - \theta^2\int_{B_1}|\nabla h|^2\,d\mu_a -(1+s)\int_{B_1}z_\infty^2\,|x_n|^a\,d\mathcal{H}^{n-1}\\
		      & -2\theta \int_{B_1} \nabla w\cdot\nabla h\,d\mu_a -4\theta \int_{B'_1}z\,\lim_{\varepsilon\to 0} \left(\varepsilon^a\,\frac{\partial h}{\partial x_n}(\wx,\varepsilon)\right)\,d\mathcal{H}^{n-1}\\
		      \stackrel{\eqref{dGG}}{=} &\int_{B_1}\!\!\!|\nabla w|^2\,d\mu_a - \theta^2\int_{B_1}\!\!\!|\nabla h|^2\,d\mu_a -(1+s)\int_{B_1}\!\!\!z_\infty^2\,|x_n|^a\,d\mathcal{H}^{n-1}
		       -2(1+s)\int_{\partial B_1}\!\!\!\!\! z_\infty\, h\,|x_n|^a\, d\mathcal{H}^{n-1}.
 \end{split}
\end{equation*}
Since $\GG_\infty^{(1)}(z_\infty)\leq \GG_\infty^{(1)}(z)$ for all $z\in \mathcal{B}_\infty^{(1)}$ then 
\begin{equation*}
 \int_{B_1} |\nabla w_\infty|^2\,d\mu_a\leq \int_{B_1} |\nabla w|^2\,d\mu_a \qquad\qquad \forall w\in w_\infty + H^1_0(B_1,\mu_a).
\end{equation*}
Using the $(1+s)$-homogeneity and \cite[Proposition 5.5]{CSS},  
the result of classification of global solutions, we deduce that 
$w_\infty=\l_\infty h_{\nu_\infty}\in \mathfrak{H}_{1+s}$ for some $\l_\infty\geq 0$ and $\nu_\infty\in\mathbb{S}^{n-2}$.\\

Thanks to \eqref{psi_j-c_j=dist(c_j,H_1+s)} we have the contradiction: from $z_j \rightharpoonup z_\infty$ in $H^1(B_1,\mu_a)$
and \eqref{def z_j} we have
\begin{equation}\label{cj/dj}
 \frac{c_j}{\d_j}=\theta_j h + z_j \to \theta h + z_\infty\in \mathfrak{H}_{1+s} \qquad\qquad \textit{in }\,H^1(B_1,\mu_a),
\end{equation}
so for $j>>1$ 
\begin{equation*}
 \mathrm{dist}_{H^1(B_1,\mu_a)}(c_j,\mathfrak{H}_{1+s})\leq \|c_j-\d_j\l_\infty h_{\nu_\infty}\|_{H^1(B_1,\mu_a)}
 \!\!\!\stackrel{\eqref{cj/dj}}{=}\!\!\! o(\d_j)<\d_j=\mathrm{dist}_{H^1(B_1,\mu_a)}(c_j,\mathfrak{H}_{1+s})
\end{equation*}
where we have used that $\d_j\l_\infty h_{\nu_\infty}\in\mathfrak{H}_{1+s}$.

\textbf{Step 6: Case (3) cannot occur.}
To prove that case (3) cannot occur, we conveniently scale  the energies so as to get a nontrivial $\Gamma$-limit for the rescaled
functionals ultimately leading to a contradiction.

By means \eqref{lim GGj=+infty}, since $\lim_j\GG_j(z_j)=+\infty$, we have
\begin{equation}\label{def gammaj}
 \gamma_j:= -4\theta_j\int_{B_1'}z_j\, \RR(h)(\wx)\,d\mathcal{H}^{n-1}\uparrow +\infty.
\end{equation}
Moreover $z_j\to z_\infty$ in $L^2(B_1')$ and \eqref{-wdh>0} give us
\begin{equation*}
 \lim_j\frac{\gamma_j}{\theta_j} = -4\lim_j \int_{B_1'}z_j\, \RR(h)\,d\mathcal{H}^{n-1}
 = -4\,\int_{B_1'}z_\infty\, \RR(h)\,d\mathcal{H}^{n-1}\in [0,\infty)
\end{equation*}
so 
\begin{equation}\label{thetajgammaj-1/2}
 \theta_j \gamma_j^{-1/2}\uparrow +\infty.
\end{equation}
Then we rescale the functional $\GG_j$ dividing by $\gamma_j$. For all $z\in\mathcal{B}_j$ we consider $\gamma_j^{-1}\GG_j(z)$ and we note that
\begin{equation}\label{tilde GGj}
 \gamma_j^{-1}\GG_j(z)=\widetilde{\GG_j}(\gamma_j^{-1/2}z)
\end{equation}
with
\begin{equation*}
   \widetilde{\GG_j}(w)= \left\{  \begin{array}{ll}
			  \displaystyle\int_{B_1}\!\!\!\!|\nabla w|^2\,d\mu_a -(1+s)\int_{\partial B_1}\!\!\!\!\!\!\!w^2\,|x_n|^a\,d\mathcal{H}^{n-1}-4\frac{\theta_j}{\gamma_j^{1/2}}\int_{B_1'}\!\!\!\!w\,\RR(h)\,d\mathcal{H}^{n-1} &  w\in \widetilde{\mathcal{B}}_j\\
			 +\infty 	 &\!\!\!\!\!\!\!\!\!\textit{otherwise},
			\end{array} \right.
\end{equation*}
where 
\begin{equation*}
 \widetilde{\mathcal{B}}_j:= \{w\in \gamma_j^{-1/2}\,z_j+H^1_0(B_1,\mu_a)\,\,:\,\, (w+\theta_j \gamma_j^{-1/2}h)|_{B_1'}\geq 0\}.
\end{equation*}
Setting $\widetilde{z}_j:=\gamma_j^{-1/2}z_j$, due to \eqref{|z_j|=1} and $\gamma_j\uparrow +\infty$, we have 
$\widetilde{z}_j\to 0$ in $H^1(B_1,\mu_a)$. Moreover the condition \eqref{tilde GGj} and the definition of 
$\gamma_j$ \eqref{def gammaj} yield
\begin{equation}\label{tilde GGj = 1+O(gammaj)}
 \widetilde{\GG_j}(\widetilde{z}_j)=\frac{\int_{B_1}|\nabla w|^2\,d\mu_a -(1+s)\int_{\partial B_1}w^2\,|x_n|^a\,d\mathcal{H}^{n-1}}{\gamma_j}+1=1+O(\gamma_j^{-1}).
\end{equation}
Thanks to \eqref{tilde GGj} we can rewrite the inequalities \eqref{quasiminimality z_j GG} as
\begin{equation*}
 (1-\kappa_j)\widetilde{\GG_j}(\widetilde{z}_j)\leq \widetilde{\GG_j}(\widetilde{z}) \qquad\qquad \forall \widetilde{z}\in \widetilde{\mathcal{B}}_j.
\end{equation*}
In particular, by taking into consideration \eqref{thetajgammaj-1/2}, $\widetilde{z_j}\to 0$ in $H^1(B_1,\mu_a)$, and \eqref{tilde GGj = 1+O(gammaj)} 
(in other words $\lim_j \widetilde{\GG_j}(\widetilde{z}_j)<\infty$) we proceed as in case (2) of Step $3$ establishing that
\begin{equation*}
 \Gamma(L^2(B_1,\mu_a))\text{-}\lim_j \widetilde{\GG_j}= \widetilde{\GG_\infty},
\end{equation*}
with
\begin{equation*}
   \widetilde{\GG_\infty}(\widetilde{z})= \left\{  \begin{array}{ll}
			  \displaystyle\int_{B_1}|\nabla \widetilde{z}|^2\,d\mu_a  & \qquad\qquad  w\in \widetilde{\mathcal{B}}_\infty\\
			 +\infty 	 &    	\qquad\qquad \text{otherwise},
			\end{array} \right.
\end{equation*}
 where $\widetilde{\mathcal{B}}_\infty=\{\widetilde{z}\in H^1_0(B_1,\mu_a)\,\,:\,\, \widetilde{z}|_{B_1'}=0\}$.
From Step $4$ and the convergence $\widetilde{z}_j\to 0$ in $H^1(B_1,\mu_a)$, the zero function turns out to be the unique minimizer of $\widetilde{\GG_\infty}$
and $\lim_j\widetilde{\GG_j}(\widetilde{z}_j)\to~\widetilde{\GG_\infty}(0)=~\!0$; this is in contradiction with \eqref{tilde GGj = 1+O(gammaj)}.\\

To prove the theorem we have only to exclude case (2) of Step $3$. 
In what follows, we suppose the hypothesis of case (2) of Step $3$: $\theta=+\infty$ and $\lim_j \GG_j(z_j)<+\infty$.
In the following steps we exhibit further properties of the limit $z_\infty$.

\textbf{Step 7: An orthogonality condition.}
By evaluating that $\psi_j$ is a point of minimal distance between $c_j$ and $\mathfrak{H}_{1+s}$, we prove that $z_\infty$ is 
orthogonal to the tangent space $T_h\mathfrak{H}_{1+s}$:

From the hypothesis $\theta=+\infty$ we deduce that $\l_j>0$ for $j>>1$. Therefore, by the condition of minimal distance \eqref{psi_j-c_j=dist(c_j,H_1+s)},
we deduce that for all $\nu\in\mathbb{S}^{n-2}$ and $\l\geq 0$, 
\begin{equation*}
 \|c_j-\psi_j\|_{H^1(B_1,\mu_a)}\leq \|c_j-\l\,h_\nu\|_{H^1(B_1,\mu_a)},
\end{equation*}
and thanks to definition of $z_j$ in \eqref{def z_j} it holds
\begin{equation*}
 \d_j\|z_j\|_{H^1(B_1,\mu_a)}\leq \|\psi_j-\l\,h_\nu+\d z_j\|_{H^1(B_1,\mu_a)}
\end{equation*}
or in the same way
\begin{equation}\label{step7 psi_j-lhnu ^2<}
 -\|\psi_j-\l\,h_\nu\|^2_{H^1(B_1,\mu_a)}\leq 2\d_j \langle z_j, \psi_j-\l h_\nu \rangle_{H^1(B_1,\mu_a)}.
\end{equation}
Now we suppose $(\l,\nu)\neq (\l_j,e_{n-1})$ and renormalizing \eqref{step7 psi_j-lhnu ^2<} we obtain
\begin{equation*}
 -\|\psi_j-\l\,h_\nu\|_{H^1(B_1,\mu_a)}\leq 2\d_j \langle z_j, \frac{\psi_j-\l h_\nu}{\|\psi_j-\l\,h_\nu\|_{H^1(B_1,\mu_a)}} \rangle_{H^1(B_1,\mu_a)}
\end{equation*}
and by passing to the limit $(\l,\nu)\to(\l_j,e_{n-1})$, reminding the definition of tangent space $T\mathfrak{H}_{1+s}$ in \eqref{def piano tangente},
we deduce
\begin{equation}\label{<z_j,zeta>geq0}
 \langle z_j,\zeta \rangle\geq 0 \qquad\qquad \zeta\in T_{\psi_j}\mathfrak{H}_{1+s}=T_h\mathfrak{H}_{1+s},
\end{equation}
where we used $\l_j>0$ in the computation of the tangent vector.
By choosing the sequence $(\l,\nu)\to(\l_j,e_{n-1})$ such that $\lim \frac{\psi_j-\l h_\nu}{\|\psi_j-\l\,h_\nu\|_{H^1(B_1,\mu_a)}}=-\zeta$
we obtain $\langle z_j,\zeta \rangle\leq 0$ thus 
\begin{equation}\label{<z_j,zeta>geq0}
 \langle z_j,\zeta \rangle= 0 \qquad\qquad \zeta\in T_h\mathfrak{H}_{1+s}. 
\end{equation}
So, taking the limit $j\to +\infty$ we conclude
\begin{equation}\label{<z_j,zeta>=0}
 \langle z_\infty,\zeta \rangle= 0 \qquad\qquad \zeta\in T_h\mathfrak{H}_{1+s}. 
\end{equation}

\textbf{Step 8: Identification of $z_\infty$ in case (2).}
There exist real constants $a_0, \dots, a_{n-2}$ such that
\begin{equation}\label{structure zinfty}
 z_\infty = a_0 h + \left(\sum_{i=1}^{n-2}a_i x_i\right) \left(\sqrt{x_{n-1}^2+x_n^2}+x_{n-1}\right)^s, 
\end{equation}
or rather $z_\infty\in T_h \mathfrak{H}_{1+s}$.

In view of homogeneity and regularity of $z_\infty$ (that is solution of a partial differential equation), 
fixed $x_{n-1}$ and $x_n$, we can write the first order Taylor polynomial 
of $z_\infty(\cdot,x_{n-1},x_n)$ in $(\underline{0}',x_{n-1},x_n)$. 
Thanks to a bidimensional argument we achieve the structure of $z_\infty$ stated in \eqref{structure zinfty}.
For its proof we refer to \cite[Lemma A.3]{GPPS} (for the reader's convenience we report a proof in Appendix).

\textbf{Step 9: Case (2) cannot occur.}
We use results of Step 4, 7 and 8 to deduce the contradiction.

From \eqref{structure zinfty} we deduce that $z_\infty\in T_h\mathfrak{H}_{1+s}$, by using it as a test function in 
\eqref{<z_j,zeta>=0}, the condition of orthogonality of Step 7 implies
\begin{equation*}
 \langle z_\infty,\zeta \rangle= 0 \qquad\qquad \zeta\in T_h\mathfrak{H}_{1+s}. 
\end{equation*}
Then we have $z_\infty=0$ but this is in contradiction with \eqref{||zinfty||=1}.

In this way we exclude the occurrence of case (2) of Step 3, thus providing the conclusion of the proof of the theorem.
\end{proof}

In what follows we show some important consequences of epiperimetric inequality.

\subsection{Decay of the boundary adjusted energy}\label{s:decay}
The following proposition establishes a decay estimate for the boundary adjusted energy.
In this connection the epiperimetric inequality allows us to estimate from below, up to a constant, the difference between 
the energy $W^{\underline{0}}_{1+s}(1,\cdot)$ evaluated respectively in the $(1+s)$-homogeneous extension of $u_{r|\partial B_1}$ and in $u_r$ 
with $W^{\underline{0}}_{1+s}(1,u_r)$; in this way we obtain a differential inequality from which we deduce the decay estimate.

\begin{prop}[Decay of the boundary adjusted energy]\label{p:decay}
Let $x_0\in \Gamma_{1+s}(u)$.
 There exists a constant $\gamma >0$ for which the following property holds:\\
 for every compact set $K\subset B_1'$ there exists a positive constant $C>0$ such that 
 \begin{equation}\label{e:decay W}
  W_{1+s}^{x_0}(r,u)\leq C\,r^\gamma,
 \end{equation}
for all radii $0<r<\mathrm{dist}(K,\partial B_1)$ and for all $x_0\in \Gamma_{1+s}(u)\cap K$.
\end{prop}
\begin{proof}
 Let us assume $x_0=\underline{0}\in\Gamma_{1+s}(u)$. 
 Thanks to Lemma \ref{l:H' e D'}, we calculate the derivative of  the boundary adjusted energy $W_{1+s}(\cdot,u)$
 \begin{equation}\label{e:deacy d/dr W inizio}
  \begin{split}
   \frac{d}{dr}&W^{\underline{0}}_{1+s}(r,u)
   = -\frac{(n+1)}{r^{n+2}}D_a(r)  + \frac{1}{r^{n+1}}D_a'(r) -\frac{(1+s)}{r^{n+2}}H_a'(r) + \frac{(1+s)(n+2)}{r^{n+3}}H_a(r)\\\
   &= -\frac{(n+1)}{r^{n+2}}D_a(r)  + \frac{1}{r^{n+1}}D_a'(r) -\frac{(1+s)(n-2s)}{r^{n+3}}H_a(r)
   - \frac{2(1+s)}{r^{n+2}}D_a(r) + \frac{(1+s)(n+2)}{r^{n+3}}H_a(r)\\
   &= -\frac{n+1}{r} W^{\underline{0}}_{1+s}(r,u) -\frac{(1+s)(n+1)}{r^{n+3}}H_a(r) + \frac{1}{r^{n+1}}D_a'(r)
   - \frac{2(1+s)}{r^{n+2}}D_a(r)  + \frac{2(1+s)^2}{r^{n+3}}H_a(r)\\
   &= -\frac{n+1}{r} W^{\underline{0}}_{1+s}(r,u) -\frac{(1+s)(n+1)}{r^{n+3}}H_a(r) + I.
   \end{split}
 \end{equation}
According to Lemma \ref{l:H' e D'} and to the definition of rescaled functions \eqref{d:ur 1+s}, we can write
\begin{equation}\label{e:decay I}
 \begin{split}
I=&\frac{1}{r^{n+2}}\int_{\partial B_r}\!\!|\nabla u|^2\,|x_n|^a\,d\mathcal{H}^{n-1} 
      + \frac{2(1+s)^2}{r^{n+3}}\int_{\partial B_r}\!\!u^2\,|x_n|^a\,d\mathcal{H}^{n-1} 
      -\frac{2(1+s)}{r^{n+2}}\int_{\partial B_r}\!\! u\nabla u\cdot \frac{x}{r}\,|x_n|^a\,d\mathcal{H}^{n-1}\\
   \stackrel{x=ry}{=}
   & \frac{1}{r} \int_{\partial B_1} \left(|\nabla u_r|^2 + 2(1+s)^2u_r^2 -2(1+s)u_r\nabla u_r\cdot y\right)\,|y_n|^a\,d\mathcal{H}^{n-1}\\
   =& \frac{1}{r} \int_{\partial B_1} \left(\left(\nabla u_r\cdot \nu -(1+s)u_r\right)^2 +|\nabla_\theta u_r|^2 + (1+s)^2u_r^2\right)\,|y_n|^a\,d\mathcal{H}^{n-1} 
 \end{split} 
\end{equation}
where by $\nabla_\theta u_r$ we denote the differential of $u_r$ in the tangent direction to $\partial B_1$.
Let $c_r$ be the $(1+s)$-homogeneous extension of $u_{r|\partial B_1}$
\begin{equation*}
 c_r(x):=|x|^{1+s}u_r\left(\frac{x}{|x|}\right).
\end{equation*}
Thus, according to $(1+s)$-homogeneity and by Euler's homogeneous function Theorem and recalling that $H_a(r)=r^{n+2}H_a(1)$ 
and $W^{\underline{0}}_{1+s}(1,u_r)=W^{\underline{0}}_{1+s}(r,u)$,
by putting together the equations \eqref{e:deacy d/dr W inizio} and \eqref{e:decay I}, 
we deduce
\begin{equation*}
\begin{split}
 \frac{d}{dr}&W^{\underline{0}}_{1+s}(r,u)= -\frac{n+1}{r} W^{\underline{0}}_{1+s}(r,u) -\frac{(n+1)(1+s)}{r}\int_{\partial B_1} u_r^2\,|x_n|^a\,d\mathcal{H}^{n-1}\\
  &+ \frac{1}{r}\int_{\partial B_1}\left(\nabla u_r\cdot \nu -(1+s)u_r\right)^2\,|x_n|^a\,d\mathcal{H}^{n-1}
    +\frac{1}{r}\int_{\partial B_1}\left(|\nabla_\theta u_r|^2+(1+s)^2u_r^2\right)\,|x_n|^a\,d\mathcal{H}^{n-1}\\
  =& -\frac{n+1}{r} W^{\underline{0}}_{1+s}(r,u)  + \frac{1}{r}\int_{\partial B_1}\left(\nabla u_r\cdot \nu -(1+s)u_r\right)^2\,|x_n|^a\,d\mathcal{H}^{n-1}\\ 
      &+\frac{1}{r}\int_{\partial B_1}\left(|\nabla_\theta c_r|^2-(1+s)(n-s)c_r^2\right)\,|x_n|^a\,d\mathcal{H}^{n-1}\\
  =&-\frac{n+1}{r} W^{\underline{0}}_{1+s}(r,u) + \frac{1}{r}\int_{\partial B_1}\left(\nabla u_r\cdot \nu -(1+s)u_r\right)^2\,|x_n|^a\,d\mathcal{H}^{n-1}\\ 
      &+\frac{1}{r}\int_{\partial B_1}\left(|\nabla c_r|^2-(1+s)(n+1)c_r^2\right)\,|x_n|^a\,d\mathcal{H}^{n-1}\\
 =&\frac{n+1}{r}W^{\underline{0}}_{1+s}(1,c_r)-\frac{n+1}{r} W^{\underline{0}}_{1+s}(1,u_r)  + \frac{1}{r}\int_{\partial B_1}\left(\nabla u_r\cdot \nu -(1+s)u_r\right)^2\,|x_n|^a\,d\mathcal{H}^{n-1}.
\end{split}
\end{equation*}
So, by Proposition \ref{p:Weiss 1+s} we have
\begin{equation*}
 \frac{d}{dr}W^{\underline{0}}_{1+s}(r,u)=2\,\frac{n+1}{r}\left(W^{\underline{0}}_{1+s}(1,c_r)-W^{\underline{0}}_{1+s}(1,u_r)\right).
\end{equation*}
 Then, according to 
 the epiperimetric inequality proved in Theorem \ref{t:epiperimetric inequality}, 
 and recalling that $u_r$ minimize $W^{\underline{0}}_{1+s}(1,\cdot)$
 we obtain
\begin{equation*}
 \frac{d}{dr}W^{\underline{0}}_{1+s}(r,u)\geq \frac{2\,(n+1)\kappa}{r(1-\kappa)}W^{\underline{0}}_{1+s}(1,u_r)=\frac{2\,(n+1)\kappa}{r(1-\kappa)}W^{\underline{0}}_{1+s}(r,u),
\end{equation*} 
 and integrating this inequality in $(0,r_0)$ we have
 \begin{equation*}
  W^{\underline{0}}_{1+s}(r,u)\leq W^{\underline{0}}_{1+s}(1,u)\,r^\gamma,
 \end{equation*}
with $\gamma:=\frac{2\,(n+1)\kappa}{r(1-\kappa)}$.
\end{proof}

\begin{oss}
 In order to prove the Proposition~\ref{p:decay} the Weiss' monotonicity formula is not necessary.
\end{oss}

\subsection{Nondegeneracy of the solution}\label{s:3nondegeneration}
In order to deduce the nondegeneracy property of the solution we note that 
the inequality \eqref{e:H(r)>r eps} is not enough. 
We state an improved version of \eqref{e:H(r)>r eps}; 
this is a consequence of epiperimetric inequality and decay estimate of energy above.

\begin{prop}[Nondegeneracy]\label{p:non degeneracy}
 Let $\uu\in H^1(B_1,\mu_a)$ be a solution of the Problem \eqref{POF'}. Let us assume that $\underline{0}\in \Gamma_{1+s}(\uu)$ .
 Then there exists a constant $H_0>0$ for which
 \begin{equation}\label{e:non degeneracy}
  H_a(r)\geq H_0\,r^{n+2} \qquad\qquad \forall 0<r<1.
 \end{equation}
\end{prop}
\begin{proof}
For the proof of this result we refer to \cite[Proposition 4.6]{FS}.
\end{proof}
By means of the nondegeneracy condition \eqref{e:non degeneracy}, for all $x_0\in \Gamma_{1+s}(u)$, we deduce
\begin{equation*}
 \int_{\partial B_1} \uu_{x_0,r}^2\, |x_n|^a\, dx \geq H_0,
\end{equation*}
and if $(\uu_{x_0,r_k})_{k\in\N}$ is a sequence that converges to $u_0$ in $L^2(B_1,\mu_a)$, a blow-up function in $x_0$, 
due to estimate \eqref{e:estimate nabla ur} and the convergence of the traces in Theorem~\ref{T:traccia partialB1} 
we obtain the convergence of the traces of $u_{x_0,r_k}$ on $\partial B_1$; thus
\begin{equation*}
 \int_{\partial B_1} \uu_0\, |x_n|^a\, dx \geq H_0>0.
\end{equation*}
So we infer $\uu_0\not\equiv 0$ for all $\uu_0$ blow-up functions in a point of $\Gamma_{1+s}(u)$.

So, in view of Propositions~\ref{p:decay}, \ref{p:non degeneracy} and \cite[Proposition 5.5]{CSS} we can deduce 
the following result of the classification of blow-ups.

\begin{prop}[Classification of blow-ups]\label{p:class blowup 1+s}
 Let $\uu$ be a solution of the Problem \eqref{POF'}. Let $u_0$ be a blow-up of $u$ in point $x_0\in \Gamma_{1+s}(u)$. Then there exist
 a constant $\l>0$ and a vector $e\in \mathbb{S}^{n-2}$ such that $u_0 = \l\, h_e$.
\end{prop}

\subsection{The blow-up method: Uniqueness of blow-ups}\label{s:3!blowup}
By summarizing what we have been showing so far, due to estimate \eqref{e:estimate nabla ur} and to Theorem~\ref{T:1.31 HKM}, 
for all $x_0\in \Gamma_{1+s}(u)$ and for all sequences
$r_k\to 0$ there exists at least a subsequence (that we do not relabel in what follows) such that 
$\uu_{x_0,r_k}\rightharpoonup \uu_{x_0}$ in $H^1(B_1,\mu_a)$ for some nontrivial functions $\uu_{x_0}\in H^1(B_1,\mu_a)$.
It is easy to prove that $u_{x_0}$ is a solution of Problem \eqref{POF'}. 
Furthermore $u_{x_0}$ is $(1+s)$-homogeneous. 
According to Proposition \ref{p:class blowup 1+s},
the result of the classification of blow-ups, we obtain $u_{x_0}\in \mathfrak{H}_{1+s}$.

With the next Proposition we prove that the blow-up is unique, i.e. for all $x_0\in \Gamma_{1+s}(u)$ there exists a function
$u_{x_0}$ such that for all $r_k\to 0$ the sequence $(u_{x_0,r_k})_{k\in\N}$ converges to $u_{x_0}$ in $L^2(B_1,\mu_a)$.
This is again a consequence of epiperimetric inequality. In particular, the epiperimetric inequality provides 
an explicit rate of convergence of the rescaled function $u_{x_0,r}$.

\begin{prop}[{\cite[Proposition 4.8]{FS}}]\label{p:uniqueness blowup 1+s}
 Let $\uu$ be a solution of the Problem \eqref{POF'} and let $K\subset\subset B_1'$. Then there exists a positive constant $C>0$ 
 such that for all $x_0\in \Gamma_{1+s}(u)\cap K$ the following inequality holds:
 \begin{equation*}
  \int_{\partial B_1} |u_{x_0,r}-u_{x_0}|\,|x_n|^a\,d\mathcal{H}^{n-1}\leq C\, r^\frac{\gamma}{2},
 \end{equation*}
where $\gamma>0$ is the constant defined in Proposition \ref{p:decay}. In particular the blow-up is unique.
\end{prop}

\section{The regularity of the free-boundary}\label{s:3reg}
Thanks to the uniqueness of blow-ups 
following the proof of  \cite[Proposition 4.10]{FS} 
it is possible to give a proof of the $C^{1,\a}$ regularity of $\Gamma_{1+s}(u)$ 
the subset of the free-boundary with lower frequency. 

\begin{teo}\label{t:reg 1+s}
 Let $\uu\in H^1(B_1,\mu_a)$ be a solution of the Problem \eqref{POF'}. Then, there exists a constant $\alpha>0$ such that for all
 $x_0\in \Gamma_{1+s}(\uu)$ there exists a radius $r=r(x_0)$ for which $\Gamma_{1+s}(\uu)\cap B_r'(x_0)$ is a $C^{1,\a}$ 
 regular $(n-2)$-submanifold in $B_1'$.
\end{teo}

\appendix
\section{Appendix}
In this Appendix we report a result of structure of a $(1+s)$-homogeneous solution of \eqref{e:PDEzinfty Appendix}
due to Garofalo, Petrosyan, Pop and Smit Vega Garcia \cite[Lemma A.3]{GPPS}. 
Recently Focardi and Spadaro in \cite[Proposition A.3]{FS16} extended this result analysing
the structure of $\l$-homogeneous solutions of \eqref{e:PDEzinfty Appendix} (also with different contact set) with 
$\l\in\{m, m+s, m+2s \,:\,  m\in \N^+, \l\geq 1+s \}$.
\begin{lem}
Let is $z_\infty$ a $(1+s)$-homogeneous solution of
\begin{align}\label{e:PDEzinfty Appendix}
		 \left\{  \begin{array}{ll}
			  L_a z_\infty= 0 & 						\quad B_1\setminus {B'}_1^{',-}\\ 
			  z_\infty= 0	  &					    	\quad {B'}_1^{,-},
			\end{array} \right.
\end{align}
even symmetric w.r.to $\{x_n=0\}$.
Then there exist real constants $a_0, \dots, a_{n-2}$ such that
\begin{equation}\label{structure zinfty}
 z_\infty = a_0 h + \left(\sum_{i=1}^{n-2}a_i x_i\right) \left(\sqrt{x_{n-1}^2+x_n^2}+x_{n-1}\right)^s, 
\end{equation}
or rather $z_\infty\in T_h \mathfrak{H}_{1+s}$.
\end{lem}
\begin{proof}
For all multi-indices $\a\in \N^{n-2}$ the derivative $\partial_\a z_\infty$ is the solution of
\begin{align}\label{e:eq diff der zinfty}
		 \left\{  \begin{array}{ll}
			  L_a \partial_\a z_\infty= 0 & 						\quad B_1\setminus {B'}_1^{,-}\\ 
			  \partial_\a z_\infty= 0	  &					    	\quad {B'}_1^{,-},
			\end{array} \right.
\end{align}
According to \cite[Lemma 2.4.1]{FKS} and \cite[Proposition 2.3]{CSS} the derivative $\partial_\a z_\infty$ are bounded in $B_{1/2}$, 
thanks to \cite[Theorems 2.3.12 and 2.4.6]{FKS} they are also continuous in $B_{1/2}\setminus \{x_{n-1}=x_n=0\}$.
We consider the second derivative $\partial_{ij} z_\infty$ with $i,j=1,\dots,n-2$: since $z_\infty$ is $(1+s)$-homogeneous, 
the function $\partial_{ij} z_\infty$ is $(s-1)$-homogeneous; as $0<s<1$ from the boundedness of the derivative we deduce 
\begin{equation}\label{e:de_ij zinfty =0}
\partial_{ij} z_\infty= 0 \qquad\qquad \textit{in } B_1\qquad\qquad \forall i,j=1,\dots,n-2.
\end{equation}
The solution $z_\infty$ is a smooth function in $B_{1/2}^+$ and $B_{1/2}^-$ because the coefficients of the strictly elliptic 
operator $L_a$ are smooth in these domains. Thus, fixed $x_{n-1}$ and $x_n$, we can write the first order Taylor polynomial 
of $z_\infty(\cdot,x_{n-1},x_n)$ in $(\underline{0}',x_{n-1},x_n)$
\begin{equation*}
 z_\infty(x',x_{n-1},x_n)= c_0(x_{n-1},x_n) + \sum_{i=1}^{n-2} c_i(x_{n-1},x_n) x_i,
\end{equation*}
with $c_0(x_{n-1},x_n)=z_\infty(\underline{0}',x_{n-1},x_n)$ and $c_i(x_{n-1},x_n)=\partial_i z_\infty(\underline{0}',x_{n-1},x_n)$.
By definition the function $c_0(x_{n-1},x_n)$ is $(1+s)$-homogeneous and the functions $c_i(x_{n-1},x_n)$ are $s$-homogeneous.
Since $z_\infty$ and $\partial_i z_\infty$ are continuous in $B_{1/2}\setminus \{x_{n-1}=x_n=0\}$ the function $c_0(x_{n-1},x_n)$ and
$c_i(x_{n-1},x_n)$ are continuous in $\mathcal{B}_{1/2}\setminus \{x_{n-1}=0\}$ 
with $\mathcal{B}_{1/2}:=\{(x_{n-1},x_n)\in \R^2\,\,:\,\, x_{n-1}^2+x_n^2<1/4\}$. 
Thanks to homogeneity with positive degree $c_0(x_{n-1},x_n)$ and $c_i(x_{n-1},x_n)$ are continuous in $\mathcal{B}_{1/2}$.\\

Taking into account \eqref{e:de_ij zinfty =0}, for all $i=1,\dots,n-2$ we obtain
\begin{equation*}
\begin{split}
 c_i(x_{n-1},x_n)=\partial_i z_\infty(x',x_{n-1},x_n)\\
 c_0(x_{n-1},x_n)=z_\infty(x',x_{n-1},x_n),
\end{split}
\end{equation*}
thus $c_i,c_0\in H^1(\mathcal{B}_{1/2}^\pm,|x_n|^a\,\mathcal{L}^2)$ and are solutions of \eqref{e:eq diff der zinfty} 
on $\mathcal{B}_{1/2}^\pm$.
Since $c_i(x_{n-1},x_n)$ is $s$-homogeneous there exist some constants $(\tilde{a}_i)_{i=1,\dots,n-2}$ such that 
$c_i(x_{n-1},0)=\tilde{a}_i x_{n-1}^s$ when $x_{n-1}>0$ and similarly since $c_0(x_{n-1},x_n)$ is $(1+s)$-homogeneous, 
there exists a constant $\tilde{a}_0$ such that $c_0(x_{n-1},0)=\tilde{a}_0 x_{n-1}^s$ when $x_{n-1}>0$.

We show that 
\begin{equation}\label{e:c_i structure zinfty}
 c_i(x_{n-1},x_n)=\frac{\tilde{a}_i}{2^s}\left(x_{n+1}+\sqrt{x_{n-1}^2+x_n^2}\right)^s.
\end{equation}
Passing to polar coordinates we can write $c_i(x_{n-1},x_n)=d_i(r,\theta)=r^s\phi_i(\theta)$. 
From $L_a c_i=0$ we deduce that the function $\phi_i$ is the solution of the following second order ordinary differential equation
\begin{equation*}
\left\{  \begin{array}{ll}
  \sin\theta \phi_{\theta\theta} + a \cos\theta \phi_\theta + (a(1+s)x+(1+s)^2)\sin\theta \phi =0   & \quad \textit{in }\,(0,\pi)\\ 
  \phi(0)=\frac{\tilde{a}_i}{2^s}&\\
  \phi(\pi)=0,&
        	\end{array} \right.
\end{equation*}
and so it has a unique solution. Resorting to a direct calculation, we can verify that the function 
\begin{equation*}
 \phi_i(\theta)=\frac{\tilde{a}_i}{2^s}(\cos\theta +1)^s 
\end{equation*}
is solution for all $\theta\in [0,\pi]$. So the function $c_i(x_{n-1},x_n)$ satisfies \eqref{e:c_i structure zinfty}.\\

By proceeding in the same way we prove that the function $c_0(x_{n-1},x_n)$ can be written as
\begin{equation*}
 c_0(x_{n-1},x_n)=\frac{\tilde{a}_0}{2^s(s-1)}\left(x_{n+1}+\sqrt{x_{n-1}^2+x_n^2}\right)^s\left(x_{n+1}-\sqrt{x_{n-1}^2+x_n^2}\right),
\end{equation*}
and this provides the conclusion to the proof of the step.
\end{proof}

\subsection*{Acknowledgments} The author is warmly grateful to Professor Matteo Focardi for his suggestion of the problem 
 and for his constant support and encouragement. The author is partially supported by project GNAMPA $2017$ 
 ``Regolarità per Problemi Variazionali Liberi e di Ostacolo''.
 The author is member of the Gruppo Nazionale per l'Analisi Matematica, la Probabilità e le loro Applicazioni (GNAMPA) of
the Istituto Nazionale di Alta Matematica (INdAM).

{\footnotesize

\vspace{0.5cm}
\noindent\textsc{DiMaI, Università degli studi di Firenze}\\
\emph{Current address:} Viale Morgagni 67/A, 50134 Firenze (Italy)\\
\emph{E-mail address:} \texttt{geraci@math.unifi.it}
}
\end{document}